\documentclass[12pt]{amsart} \usepackage{amsmath, amssymb} \newfont {\cyr} {wncyr10} 
\usepackage{amsfonts} \usepackage{amsmath} \usepackage{amssymb} \usepackage{setspace}
\usepackage{multicol}
\usepackage{color}
\usepackage{caption}

\usepackage{mathtools}
\usepackage{fancyref}

\newcommand*{\fancyrefthmlabelprefix}{thm}
\frefformat{plain}{\fancyrefthmlabelprefix}{Theorem \textup{#1}}

\newcommand*{\fancyrefequationlabelprefix}{eq}
\frefformat{plain}{\fancyrefequationlabelprefix}{\textup{(#1)}}

\newcommand*{\fancyrefsectionlabelprefix}{sec}
\frefformat{plain}{\fancyrefsectionlabelprefix}{Section~\textup{#1}}

\newcommand*{\fancyrefnotlabelprefix}{not}
\frefformat{plain}{\fancyrefnotlabelprefix}{Notation~\textup{#1}}

\newcommand*{\fancyrefhyplabelprefix}{hyp}
\frefformat{plain}{\fancyrefhyplabelprefix}{Hypothesis~\textup{#1}}

\newcommand*{\fancyreflemmalabelprefix}{lem}
\frefformat{plain}{\fancyreflemmalabelprefix}{Lemma~\textup{#1}}

\newcommand*{\fancyrefdeflabelprefix}{def}
\frefformat{plain}{\fancyrefdeflabelprefix}{Definition~\textup{#1}}

\newcommand*{\fancyrefproplabelprefix}{prop}
\frefformat{plain}{\fancyrefproplabelprefix}{Proposition~\textup{#1}}

\newcommand*{\fancyrefcorlabelprefix}{cor}
\frefformat{plain}{\fancyrefcorlabelprefix}{Corollary~\textup{#1}}

\renewcommand*{\fancyreftablabelprefix}{tab}
\frefformat{plain}{\fancyreftablabelprefix}{Table~\textup{#1}}

\newcommand*{\fancyrefremlabelprefix}{rem}
\frefformat{plain}{\fancyrefremlabelprefix}{Remark~\textup{#1}}

\newcommand*{\fancyrefchplabelprefix}{chp}
\frefformat{plain}{\fancyrefchplabelprefix}{Chapter~\textup{#1}}

\newcommand*{\fancyrefclaimlabelprefix}{claim}
\frefformat{plain}{\fancyrefclaimlabelprefix}{Claim~\textup{#1}}

\newcommand*{\fancyrefclmlabelprefix}{clm}
\frefformat{plain}{\fancyrefclmlabelprefix}{\textup{#1}}

 \newtheorem{Def}{Definition}[section]
\newtheorem{hyp}[Def]{Hypothesis}
\newtheorem{lemma}[Def]{Lemma}

\newtheorem*{CD*}{Case Divison}
 \newtheorem{remark}[Def]{Remark}
\newtheorem{proposition}[Def]{Proposition} 
 \newtheorem{Equ}{} 
 \newtheorem{notation}[Def]{Notation}

\newcounter{claim}[Def]

\newtheorem*{theorem*}{Main Theorem}
\newtheorem*{corollary*}{Corollary}

\newcommand{\Z}{{\mathbb Z}}
\newcommand{\Q}{{\mathds Q}}

 \def \Aut {\mbox {\rm Aut}}       \def \Syl {\mbox {\rm Syl}}  \def \Q {\mbox {\rm Q}}\def \Out {\mbox {\rm Out}} \def \Fun {\mbox {\rm Fun}}

\def \OO {\mathrm O}

\mathchardef\tnode="020E 

\def\arc{
  \hbox{\kern -0.15em
  \vbox{\hrule width 2.5em height 0.6ex depth -0.5 ex}
  \kern -0.33em}}

\def\darc{
  \rlap{\lower0.2ex\arc}{\raise0.2ex\arc}}

\def\tarc{
 \rlap{\rlap{\lower0.4ex\arc}{\raise0.4ex\arc}}{\arc}}

\def\stroke#1{
  \kern 0.05em
  \rlap\arc{{\textstyle{#1}}\atop\phantom\arc}
  \kern -0.22em}

\def\dstroke#1{
  \kern 0.05em
  \rlap\darc{{\textstyle{#1}}\atop\phantom\darc}
  \kern -0.22em}

\def\etc{\>\>\cdots\>\>} 

\def\centerscript#1{
  \setbox0=\hbox{$\tnode$}
  \hbox to \wd0{\hss$\scriptstyle{#1}$\hss}}

\def\node{
  \def\super{}
  \def\sub{}
  \futurelet\next\dolabellednode}

  \let\sp=^
  \let\sb=_

  \def\dolabellednode{%
    \ifx\next\sb\let\next\getsub
    \else
      \ifx\next\sp\let\next\getsuper
      \else\let\next\donode
      \fi
    \fi
    \next}

  \def\getsub_#1{\def\sub{#1}\futurelet\next\dolabellednode}

\def\getsuper^#1{\def\super{#1}\futurelet\next \dolabellednode}

  \def\donode{%
\rlap{$\mathop{\phantom\tnode}\limits_{\centerscript{\sub}}
    ^{\centerscript{\super}}$}\tnode}

\def\varcdn{
  \kern -0.03em\vbox{\kern -0.5ex
  \hbox to \wd0{\hss\vrule width 0.04em depth 5.8ex\hss}
  \kern -0.3ex
  \hbox{$\tnode$}}}

\def\m24{\node\stroke{c}\node\darc\node\arc\node} \def\tc5{\node\arc\node\arc\node\arc\node\dstroke{\sim}\node}

\def\Co1{\node\arc\node\arc\node\dstroke{\sim}\node} 

\newcommand{\varcdnl}[1]{ 
  \kern 0.3em\vbox{\kern -0.5ex
  \hbox to \wd0{\hss\vrule width 0.04em depth 5.8ex\hss}
  \kern -1ex
  \hbox{$\tnode^{#1}$}}}

\def\cnlchap7{\node^1\arc\node^2\arc\node^3\etc\node \arc \node \darc \node}
\def\dnlchap7{\node^1\arc\node^2\arc\node^3\etc\node \arc \node_{\varcdnl{}}\arc\node}

 \def\m24l{\node^1\stroke{c}\node^2\darc\node^3\arc\node^4}
\def\tc5l{\node^1\arc\node^2\arc\node^3\arc\node^4\dstroke{\sim}\node^5}
\def\U43l{\node^1\darc\node^2\darc\node^3} \def\F2l{\node^1\arc\node^2\arc\node^3\stroke{M_{22}}\node^4}
 
 \def\Co2l{\node^1\arc\node^2\stroke{M_{22}}\node^3}
\def\J4l{\node^1\arc\node^2\stroke{{\widetilde{M_{22}}}}\node^3}
\def\Dvier{\node^3\arc\node_{\varcdnl{4}}^1\arc\node^2} \def\sm24{\node^2\arc\node^1\dstroke{\sim}\node^0}
\def\sCo1{\node^4\arc\node^2\arc\node^1\dstroke{\sim}\node^0}

\def\nodef{
  \def\super{}
  \def\sub{}
  \futurelet\next\dolabellednodef}

  \let\sp=^
  \let\sb=_

  \def\dolabellednodef{%
    \ifx\next\sb\let\next\getsubf
    \else
      \ifx\next\sp\let\next\getsuperf
      \else\let\next\donodef
      \fi
    \fi
    \next}

\def\getsubf_#1{\def\sub{#1}\futurelet\next\dolabellednodef}

\def\getsuperf^#1{\def\super{#1}\futurelet\next \dolabellednodef}

  \def\donodef{%
\rlap{$\mathop{\phantom\tnodef}\limits_{\centerscript{\sub}}
    ^{\centerscript{\super}}$}\tnodef}

\def\varcdnf{
  \kern -0.03em\vbox{\kern -0.5ex
  \hbox to \wd0{\hss\vrule width 0.04em depth 5.8ex\hss}
  \kern -0.3ex
  \hbox{$\tnodef$}}}

\newcommand{\e}{\end{document}}

\def\multdef#1#2{\actdef{#1}{#2}}

\def\actdef#1#2#3{%
   \if #3. \def\next{}%
   \else\expandafter\def\csname #1#3\endcsname{{#2{#3}}}%
        \def\next{\multdef{#1}{#2}}
   \fi%
 \next}

\multdef{}{\operatorname}
          {lcm}{Hom}{B}{Aut}{Fix}{Inn}{GF}{Ly}
          {McL}{Sym}{Mat}{Cut}{Jump}{Alt}{Stab}
          {Syl}{D}{SD}{rad}{sgn}{id}{End}{tr}
          {Frob}{Ann}{rank}{Span}{J}{Int}
          {Irr}{IBr}{Bl}{Br}{Dih}{Supp}{FSym}
          {lead}{Diag}{GL}{SL}{ad}{orb}{T}{M}{nat}{PSL}
          {PSU}{PSp}{Sp}{G}{SU}{Sz}{F}{GU}{J}{U}{Co}{Suz}{E}{F}{L}{HN}{He}{Spin}{SO}
{PGL}.
\def \POmega {\mathrm{P}\Omega}

\begin{document}
 \title {The saturated fusion systems on a Sylow 2-subgroup of $\Omega^+_8(2)$}
\author{ Gernot Stroth }
\address{Gernot Stroth\\
Institut f\"ur Mathematik\\ Universit\"at Halle - Wittenberg\\
Theodor Lieser Str. 5\\ 06099 Halle\\ Germany}
\email{gernot.stroth@mathematik.uni-halle.de}
\maketitle
\begin{abstract} We consider  saturated fusion systems $\mathcal F$ on a Sylow $2$-subgroup of $\Omega^+_8(2)$ with $O_2(\mathcal F) = 1$. Examples for this are the $2$-fusion systems of $\Omega^+_8(2)$, $\Omega^+_8(2):3$, $\POmega^+_8(3)$ and $\POmega^+_8(3):3$.
\end{abstract}
\vspace{0.5cm}

In \cite[Conjeture 2]{PaSem} the authors raised the following conjecture: {\it Suppose that $p$ is a prime and $\mathcal S$ is a collection of pairs $(G,S)$ where $G$ is a simple group of Lie type defined in characteristic $p$, which is not isomorphic to $\PSp_4(p^{a})$ and $S$ a Sylow $p$-subgroup of $G$. Then for all but finitely many exceptions  if $(G,S) \in \mathcal S$ and $\mathcal F$ is a saturated fusion system on $S$ with $O_p(\mathcal F) = 1$, then $\mathcal F = {\mathcal F}_S(H)$ for some $G \leq H \leq \Aut(G)$.}
\\
\\
I believe that if the rank and the field are large enough $G$ is uniquely determined among the simple groups with Sylow $p$-subgroup $S$.
Hence one could also say that {\it if rank and field are large enough there are no exotic fusion systems on $S$}. \\
\\
In this paper we will concentrate at the case $q$ even. The case $q$ odd certainly will need different methods. If the Lie rank of $G$ is two this has been dealt with in \cite{vB1}. So let us assume that the Lie rank of $G$ is at least 3. To me no case is known in which the Sylow 2-subgroup of $G$ for $q > 2$ and Lie rank at least three  is isomorphic to a Sylow 2-subgroup of some different simple group. Hence we may concentrate on $q = 2$. If $|S| \leq 2^9$ this was dealt with in \cite{An1}. Thus we may assume that $|S| \geq 2^{10}$. We have the famous three simple groups $\L_5(2)$, $\M_{24}$ and $\He$ with the same Sylow 2-subgroup. These fusion systems have been dealt with in \cite{OlivVen}. An easy inspection shows that there is just the Sylow 2-subgroup of $\Omega^+_8(2)$ left, which is also a Sylow 2-subgroup of $\POmega^+_8(3)$. The purpose of this paper is to investigate the saturated fusion systems on the Sylow 2-subgroup of $\Omega^+_8(2)$.  We claim (believe) that all the counterexamples in characteristic two to the conjecture above have been done so far.
\\
\\
The groups  $\Omega_8^+(2)$, $\Omega^+_8(3)$, $\Omega_8^+(2):3$ and $\Omega^+_8(3):3$ all have a Sylow $2$-subgroup of type $\Omega^+_8(2)$. The $2$-fusion systems are pairwise not isomorphic.
We will prove:
\begin{theorem*}\label{thm:fusion} If $\mathcal F$ is a saturated fusion system on a Sylow $2$-subgroup of $\Omega_8^+(2)$ with $O_2(\mathcal F) = 1$, then $O^{2^\prime}(\mathcal F)$ is the $2$-fusion system of $\Omega^+_8(2)$ or $\POmega^+_8(3)$.
\end{theorem*}

The proof also yields the $2$-fusion system for $\POmega^+_8(p)$, where $p \equiv \pm 3 (\mod 8)$, which is the same as the one for $\POmega^+_8(3)$. As a corollary we obtain:
\begin{corollary*} If $\mathcal F$ is a reduced saturated fusion system on a Sylow $2$-subgroup of $\Omega_8^+(2)$, then $\mathcal F$ is the $2$-fusion system of $\Omega^+_8(2)$ or $\POmega^+_8(3)$.
\end{corollary*}

We will now sketch the proof. The main idea of the proof is not to use direct calculations between elements of $S$ or MAGMA calculations. We will use more substantial properties, which have the chance to be generalized if one tries to deal with Sylow 2-subgroups of groups of Lie type over fields $\GF(2^a)$, $a > 2$.

Let $\mathcal F$ be a saturated fusion system on a $p$-group $S$.
The basic result will be the  Alperin-Goldschmidt-Fusion-Theorem \cite[Theorem I.3.5]{AKO} (see also \fref{prop:alperin}) which says that $$\mathcal F = \langle \Aut_{\mathcal F}(E), \Aut_{\mathcal F}(S) ~|~E \text{ essential subgroup of }S\rangle.$$

Hence the first step is to determine the essential subgroups of $S$. In a group of Lie type by the Borel-Tits Theorem \cite[Theorem 3.1.3]{GLS} the fusion is determined by the minimal parabolics containing a Sylow $p$-subgroup. Hence in a first step we try to show that essential subgroups $E$ have to be normal in $S$.

For this we fix some essential subgroup $E$, which is not normal in $S$. Recall that $C_S(E) \leq E$, $O_2(\Aut_{\mathcal F}(E)/\Inn(E)) = 1$ and $\Aut_{\mathcal F}(E)/\Inn(E)$ contains a strongly $2$-embedded subgroup. In fact we do not use the full power of being essential. We just use that $C_S(E) \leq E$ and there is a certain subgroup of $N_{\mathcal F}(E)/E$, which has a strongly 2-embedded subgroup and as Sylow 2-subgroup the group $N_S(E)/E$.

There is some distinguished subgroup $Q$ in $S$, the radical of the normalizer of $Z(S)$. This is an extra-special group with $Q^\prime = Z(Q) = \Phi (Q) = Z(S)$. It is easy to see that $E \not\leq Q$, as $Q/E$ cannot embed in a strongly $2$-embedded subgroup (see \fref{lem:normal1}). That $Q \not\leq E$ in our case belongs to the fact hat $S/Q$ is abelian and so any overgroup of $Q$ is normal in $S$.

Set $T = N_S(E)$. Then in \fref{lem:normal2} we show that $Q \not\leq T$. We first get rid of the cases that $|T/E| \geq 4$. For this we use the possible actions of strongly $2$-embedded subgroups on $E$ as given in \fref{sec:prem}.
Unfortunately if $|T:E| = 2$, we always have a strongly $2$-embedded subgroup in $\Out(E)$. Here different methods are necessary.

We now follow the line in \cite{An}. We consider some $g \in N_S(T) \setminus T$ with $g^2 \in T$, recall $T \not= S$.
By the remark before we may choose $g \in Q$, which simplifies the action on $E$. Now there is some group $G_E$ which realizes $N_{\mathcal F}(E)$. Further there is $G_{E^g}$, which is isomorphic to $G_E$.


 We receive by \cite[Lemma 4.1]{An} groups $G_1 \leq G_E$ and $G_2 \leq  G_{E^g}$ such that $G_i/O_2(G_i) \cong D_{2p}$ for some odd prime $p$ and $g$ induces an automorphism $\beta_g$ between $G_1$ and $G_2$, which agrees with $g$ on $T$. We have $G_1/E \cong D_{2p}$ for some odd prime $p$.   Let $N$ be the largest normal subgroup of $T$ which is normal in both $G_1$ and $G_2$. Then $N^g = N$ and $(G_1/N, G_2/N)$ forms an amalgam on which $g$ acts. As $g \in Q$, we have a very restricted action.

We first investigate the case that $O_p(G_1/N) \not= 1$. Then $G_1/N \cong \Z_2 \times D_{2p}$ and $N$ is centric. Now $N_Q(E) \not\leq O_2(G_1)$ and we can determine the action on $E \cap Q$, which then eventually yields a contradiction (see  \fref{lem:nonnonstrict} and \fref{prop:strict}).

Now we have that $(G_1/N, G_2/N)$ is a Goldschmidt/Fan-amalgam.   Then we can apply \cite{Fan,Gold} and due to the fact that $G_1/N \cong G_2/N$ there are just two cases left, the amalgam of type $\L_3(2)$ and the one of type $\Sp_4(2)$.  In $T/N$ there are exactly two maximal elementary abelian subgroups which then are $E/N$ and $E^g/N$. Now $N_S(T)$ has to interchange these two groups.
As $g \in Q$, we have $[E,g] \leq Q$. The structure of $T/N$, then implies that $[E,g]$ contains an element of order four and so $Z(Q) \cap N = 1$, $Q \cap N$ is elementary abelian. With all these we eventually  reach at a contradiction and so we have that essential subgroups are normal in $S$, \fref{prop:essential}.

Now again we try to prove the analogue of the Borel-Tits Theorem. In a group of Lie type we have that $N_{\mathcal F}(E)$ is a minimal parabolic, i.e. $O^{2^\prime}(N_{\mathcal F}(E)/O_2(N_{\mathcal F}(E)))$ is a rank 1 group of Lie type. This in fact we are able to prove (\fref{prop:normalinS}), rank 1 group now means $\Sigma_3 \cong L_2(2)$. 

But in the actual case there might be one further essential normal subgroup of $S$ (see \fref{lem:essentialgroups}), which does not corresponds to a minimal parabolic of $\Omega^+_8(2)$.  This is the split between the $2$-fusion system of $\Omega^+_8(2)$ (just the radicals of the minimal parabolics of $\Omega_8^+(2)$) and $\POmega^+_8(3)$ (one essential subgroup more than the radicals of the minimal parabolics of $\Omega^+_8(2)$).

 Up to this point we do not use any properties of $\mathcal F$. In particular without using $O_2(\mathcal F) = 1$, we get all the possibilities for the essential subgroups. This might be of interest if one tries to deal with fusion systems on Sylow 2-group of $\Omega^+_{2n}(2)$, $n > 4$, where one could use induction.

 We also try to put concrete calculations to a minimum. What we use is the existence of an (extra)-special group $Q$, which is weakly $\mathcal F$-closed in $S$ and some elementary abelian subgroups normal in $S$. All this can be generalized to $q = 2^a > 2$, while calculations probably cannot be done. If one wants to deal with the cases $q = 2^a§´$, then one should prove a version of  \cite[Lemma 4.1]{An}.

Afterwards we can approach via the suggestions due to S. Onofrei \cite{Orno}, which also would work for odd primes $p$. There to any essential subgroup $E$ we attach a group $G_E$ and the $G_E$ at the end build a parabolic system, which yields a chamber system whose universal cover is known.  In  our case of exactly  four essential subgroups we get the building of type $D_4(2)$  by \cite{Tim}, and for five essential subgroups we receive by \cite{Kant} the building of $\OO_8(\Q_2,f)$ with discrete acting group $\Omega_8(\Z[\frac{1}{2}],f)$. Set $\tilde{\mathcal F} = \langle O^{2^\prime}(\Aut_{\mathcal F}(E))~|~E \text{ essential in }S\rangle$. Then in the former we receive that $\tilde{\mathcal F}$ is the $2$-fusion system of $\Omega^+_8(2)$. In the second case we now see that $\tilde{\mathcal F}$ is the $2$-fusion system of $\Omega_8(\Z[\frac{1}{2}],f)$. Further as $\Omega_8(\Z[\frac{1}{2}],f)$ has a finite faithful image $\POmega^+_8(3)$, we get that $\tilde{\mathcal F}$ is the $2$-fusion system of $\POmega_8^+(3)$. But as also $\POmega^+_8(p)$ for $p \equiv \pm 3 (\mod 8)$ are finite images they all have the same $2$-fusion system.  But if $q$ is not a prime this group is not generated by the normalizers of the essential subgroups.

As $\mathcal F = \langle \tilde{\mathcal F}, \Aut_{\mathcal F}(S) \rangle$ we see that $\tilde{\mathcal F} = O^{2^\prime}(\mathcal F)$ (\fref{lem:subsystem}). This then proves the main theorem.

One further remark is in order. $\Omega_8^+(2)$ possesses an automorphism of order three permuting the three essential subgroups in $S$.
If $N_{\mathcal F}(S) \not= S$, then also $\mathcal F$ could be the 2-fusion system of $\Omega^+_8(2):3$. As a Sylow $3$-subgroup of $\Aut(S)$ is of order three (\fref{lem:aut}) we see that the action of an element of order three in $N_{\mathcal F}(S)$ is as the one of $\Aut(\Omega^+_8(2))$ or $\Aut(\POmega^+_8(3))$ on the corresponding $2$-fusion system.

But this is not a general approach which would work if we consider automorphism groups  of groups of Lie type involving graph automorphisms. In an appendix we will give a different approach, which is based on a result of \cite{BMOR} which gives that in our situation $\mathcal F$ is not exotic. Then we may use the classification of the finite simple groups to find out that in the case just described for a realization we really have a normal subgroup of index three. More precise

{\bf \fref{prop:realization}} {\it Let  $\mathcal F$ be a fusion system on a Sylow $2$-subgroup of $\Omega_8^+(2)$ with $O_2(\mathcal F) = 1$.  Then $\mathcal F$ is the $2$-fusion system of a group $H$ with $F^\ast(H) \cong \Omega_8^+(2)$ or $\POmega^+_8(3)$ (more concrete  $\Omega_8^+(2) \leq H \leq \Omega^+_8(2):3$ or  $\POmega_8^+(3) \leq H \leq \POmega^+_8(3):3$).}
\\
\\
We state this result independently because in contrast to the results in \fref{sec:prem}-\fref{sec:F} the proof of \fref{prop:realization} uses the classification of the finite simple groups.

\section{Preliminaries}\label{sec:prem}
Let $\mathcal{F}$ be some fusion system on a $p$-group $P$. For the basic definitions and properties of fusion systems we refer to the paper \cite{AKO}.  The $\mathcal F$-essential subgroups and their automizers play a prominent role for the structure of $\mathcal{F}$.
Hence we will repeat the definition as far as it is important for this paper.

\begin{Def}\label{def:essential}{\rm (\cite[Definitions I.3.1, I.3.2]{AKO})} {\rm Let $\mathcal F$ be a fusion system over a finite $p$-group $S$. A subgroup $P$ of $S$ is called
\begin{itemize}

\item fully $\mathcal F$-centralized, if $|C_S(\phi P)| \leq |C_S(P)|$ for all $\phi \in \Hom_{\mathcal F}(P,S)$;
\item fully $\mathcal F$-normalized, if $|N_S(\phi P) \leq  |N_S(P)|$ for all $\phi \in \Hom_{\mathcal F}(P,S)$;
\item $\mathcal F$-centric, if $P$ is fully centralized and $C_S(P) = Z(P)$;
\item  $\mathcal F$-essential, if $P$ is $\mathcal F$-centric, fully $\mathcal F$-normalized and $\Out_{\mathcal F}(P)$ has a strongly $p$-embedded subgroup.
\end{itemize}}
\end{Def}

A subgroup $M$ of $\Out_{\mathcal F}(P)$ is strongly $p$-embedded if $M$ contains a Sylow $p$-subgroup $R$ of $\Out_{\mathcal F}(P)$, $R \not= 1$, and $R \cap R^\phi = 1$ for any $\phi \in \Out_{\mathcal F}(P) \setminus M$.
\\
\\
 We usually drop the prefix $\mathcal F$. The importance of this definition comes from the Alperin-Goldschmidt-fusion-theorem:

\begin{proposition}\label{prop:alperin}{\rm (\cite[Theorem I.3.5]{AKO})} Let $S$ be a $p$-group and $\mathcal F$ be a saturated fusion system over $S$. Then
$$\mathcal F = \langle \Aut_{\mathcal F}(P)~|~P = S \text{ or }P \text{ is } \mathcal F-essential \text{ in } S\rangle.$$
\end{proposition}

In this paper $p = 2$. From the property centric we just use $C_S(P) = Z(P)$.
To get a list of possible essential subgroups we determine those $P$ with $C_S(P) \leq P$, which admit a group $X$ of outer automorphisms which contains a strongly 2-embedded subgroup such that $N_S(P)/P$ is a Sylow 2-subgroup of $X$. The last one is what we use from fully normalized. Sometimes such groups $P$ are also called {\it critical}.
\\
\\
In the first place we have to know the candidates for $X$. This is  the next lemma.

\begin{lemma}\label{lem:strongembedded}{\rm (\cite{Be})} Let $G$ be a finite group with a strongly $2$-embedded subgroup. Then one of the following holds
\begin{itemize}
\item[(i)] A Sylow $2$-subgroup $T$ of $G$ is cyclic or quaternion; or
\item[(ii)] $O^{2^\prime}(G/O(G)) \cong \PSL_2(q)$, $\Sz(q)$ or $\PSU_3(q)$, $q$ even, $q > 2$.
\end{itemize}
\end{lemma}

Let us consider case \fref{lem:strongembedded}(i). If the Sylow 2-subgroup $T$ is cyclic then by Burnside \cite[Hauptsatz IV.2.6]{Hu} we have a normal 2-complement, whence $G = O(G)T$.

If $T$ is quaternion, then by the Brauer- Suzuki-Theorem \cite{BrauSuz} (for a complete proof see \cite[Theorem 6.8, Theorem 14.11]{Dade}) $G = O(G) C_G(Z(T))$. In both cases consider $O(G)T$. Then by the Frattini argument $T$ normalizes
some $p$-group in $O(G)$, which is not centralized by $Z(T)$. Hence in case (i) we have a subgroup (with a strongly 2-embedded subgroup) which is an extension of a $p$-group $P$ by $T$ such that $T$ acts irreducibly on $P/\Phi(P)$ and $\Omega_1(T)$ inverts $P/\Phi(P)$.

Recall that we do not claim that in the quaternion case $G$ is solvable as it is in the cyclic case. An extension of an elemenary abelian group of order 25 by $\SL_(5)$ with natural action would be an example for \fref{lem:strongembedded}(i).
\\
\\
In case of \fref{lem:strongembedded}(ii) there is a subgroup, which is an extension of a group of odd order by $\PSL_2(^n)$, $\Sz(2^n)$ or $\PSU_3(2^n)$, $n \geq 2$.
\\
\\
When we will study those subgroups $E$ of the Sylow 2-subgroup $S$, which are candidates for essential subgroups we do not use all the properties in full generality. We use in the first place that $C_S(E) \leq E$. Then we assume that $N_S(E)$ is a full Sylow 2-subgroup of $N_{\mathcal F}(E)$. But in fact this we do not really use, what we use is that there must be a group $U$ such that $E \unlhd U$ and $T = N_S(E)/E$ is a Sylow 2-subgroup of $U/E$. For an essential subgroup we could assume that $U$ possesses a strongly $2$-embedded subgroup.
But what we assume instead is that $U$ is one of the groups just mentioned, i.e.
\begin{itemize}
\item $U$ is nonsolvable and $U/O(U) \cong \PSL_2(^n)$, $\Sz(2^n)$ or $\PSU_3(2^n)$, $n \geq 2$.
\item $U$ is solvable, $U = O_p(U)T$ for some prime $p$, $T$ is cyclic of order at least four or quaternion  and  acts irreducibly on $O_p(U)/\Phi(O_p(U))$
with $\Omega_1(T)$ inverting this group.
\item $|N_S(E)/E| = 2$.
\end{itemize}

To use this information and to show that in our case just the third variant above occurs, we have to have informations about a possible action of $U$ on $E$. This will be done in the following two lemmas. The first one is the non solvable case, while the second one will be the solvable case. The restriction on $[V,t]$ in \fref{lem:representation1} follows from special properties of $S$, in fact it is due to $|S/Q| = 8$, where $Q$ is the normal extraspecial group of order $2^9$. The following results also can be found in a similar form in \cite[Section 1.2]{OlivVen}.

\begin{lemma}\label{lem:representations} Let $X/O_{2,2^\prime}(X) \cong Sz(q)$, $\PSU_3(q)$ or $\PSL_{2}(q), q$ even, $q >2$, $X = X^\prime$, and $V$ be some
faithful $\GF(2)$-module for $X$.
\begin{itemize}
\item[(i)] If $X/O_{2,2^\prime}(X) \not\cong \PSU_3(q)$. We have
$|[V,t]|\geq q^{\, 2}$, $q$, respectively, where $t$ is any involution
in $X$. Furthermore if $|[V,t]|\leq 8$, then there is just one nontrivial irreducible $X$-module involved in $V$.
\item[(ii)] $|V| \geq q^6$ in case of $X/O_{2,2^\prime}(X) \cong \PSU_3(q)$.
\end{itemize}
\end{lemma}

\begin{proof} (i) We have that $t$ inverts an element of order
$q+\sqrt{2q}+1, q+1$, respectively, hence $|[V,t]|\geq q^{\, 2}, q$ by \cite[(XI 3.10)]{Hu3}.

(ii) As $q > 2$ there is be some Zsygmondy prime $p$, which divides $q^6-1$. Then, as $q^3+1$ divides $|X|$ and by definition of a Zsygmondy prime $p$ does not divide $2^{t}-1$ for $t < 6n$, we see that the smallest $\GL_m(2)$ whose order is divisible by $p$ would be $\GL_{6n}(2)$, the assertion.
\end{proof}

\begin{lemma}\label{lem:representation1} Let $X$ be a group, $p$ be an odd prime such that  $X/O_p(X)$ is cyclic of order at least four or quaternion. Assume further that $X$ acts irreducibly and faithfully on $O_p(X)/\Phi(O_p(X))$. Let  $V$ be a faithful $\GF(2)$-module for $X$ and $t$ be an involution in $X$. Assume that  $1 \not= |[V,t]| \leq 8$, then one of the following holds
\begin{itemize}
\item[(i)] $p = 5$, $|O_5(X)| = 5$ and $|[V,O_5(X)]| = 16$, $X/O_5(X) \cong \Z_4$.
\item[(ii)] $p = 3$, $O_3(X)$ is elementary abelian of order $9$ and $X/O_3(X) \cong \Z_4$, or extraspecial of order $27$ and $X/O_3(X) \cong \Z_4$, $\Z_8$ of $Q_8$. In all cases $|[V,O_3(X)]| \leq 64$ with equality in the last three cases. In the extraspecial case $|[V,O_3(X),t]| = 4$.
\end{itemize}
 \end{lemma}

\begin{proof}  Recall that by assumption $t$ inverts $O_p(X)/\Phi(O_p(X))$. If $O_p(X)$ is cyclic, then $O_p(X) = \langle x \rangle$ with $x^t = x^{-1}$. Now $x \in \langle t, t^x \rangle$ and then $O_p(X) \leq \GL_6(2)$, which gives that $|O_p(X)| = p$ or $9$. But as $O_p(X)$ has to admit a cyclic group of order 4, we get $O_p(X) = O_5(X)$ is of order 5 and we have (i).

Assume next $|O_p(X)/\Phi (O_p(X))| \geq  p^2$. Then there are $x_1, x_2 \in O_p(X)$ with $|\Phi(O_p(X))\langle x_1, x_2 \rangle/\Phi(O_p(X))| = p^2$ and $x_i^t = x_i^{-1}$.
 Furthermore we have that $|[V,x_i]| \leq 64$ and so as before $o(x_i) = 3, 5, 7$ or $9$.  We first show
 \begin{gather*}\label{eq:1} p\not= 5\tag{1}\end{gather*}
 Assume $p = 5$. Then $|[V,x_i]| = 16$ and $|[V,t] \cap [V,x_i]| = 4$. Thus $|[V,\langle x_1, x_2\rangle]| \leq 64$. But $\GL_6(2)$ contains no non cyclic 5-subgroups. This proves \fref{eq:1}.
\\
\\
 Let $y$ be of order four with $y^2 = t$. By \fref{eq:1} $y$ cannot normalize $\Phi(O_p(X))\langle x_1 \rangle$. Thus we may chose notation such that $x_1^y = x_2$. Again $|[V,x_i]| \leq 2^6$, $i = 1,2$.
 \begin{gather*}\label{eq:2} |[V,x_1]| \leq 2^4. \text{ In particular }p=3.\tag{2}\end{gather*}
 Assume $|[V,x_1]| = 2^6$.  Then $[V,t] \leq [V,x_i]$. Set $W =[V,\langle x_1, x_2\rangle]$, $|[W,t]| = 8$. Then $|W| \leq 2^9$. Suppose $\langle x_1, x_2 \rangle$ induces an elementary abelian group on $W$. Then $x_2$ acts on $C_W(x_1)$. But $t$ centralizes $C_W(x_1)$, which implies $[C_W(x_1),x_2] = 1$ and so $[V,x_1]=[V,x_2]$ is of order $2^6$. Hence $\langle x_1,x_2 \rangle$ induces a subgroup of $\GL_6(2)$. Let first $p = 3$. As $[V,x_1]$ is generated by the centralizers of the 3-elements in $\langle x_1, x_2 \rangle$, i.e. $[V,x_1] = C_{[V,x_1]}(x_1x_2) \times  C_{[V,x_1]}(x_1(x_2)^{-1})$ and $(x_1x_2)^{y^{-1}} = x_1x_2^{-1}$ this contradicts $|[V,x_1]| = 2^6$. Thus we have $p = 7$. But then for all $x \in \langle x_1,x_2\rangle$ we have that $[V,x_1] = [V,x]$ as $x$ is inverted by $t$. But this contradicts coprime action.

 We have that $\langle x_1, x_2 \rangle$ induces a non abelian subgroup  of $\GL_9(2)$. This implies  $p = 3$. Further as a Sylow 3-subgroup of $\GL_9(2)$ has a fixed point on the natural module $\langle x_1, x_2 \rangle$ induces a subgroup $Y$ of $\GL_8(2)$. In particular it is contained in $U = \Z_3 \wr \Z_3 \times \Z_3$. We consider $Z(Y) \cap \Phi(Y)$ which is contained in $\Phi(U)$  and so in  $\GL_6(2)$. Then $|Z(Y) \cap \Phi(Y)| = 3$. As $y$ normalizes $Z(Y) \cap \Phi(Y)$, we see that $[Z(Y) \cap \Phi (Y), t]=1$. This implies that $|[W,Z(Y) \cap \Phi(Y),t]| = 4$. As $|[W,t]| = 8$,  $t$ induces a transvection $C_W(Z(Y) \cap \Phi(Y))$. As $t = y^2$ this implies that $[C_W(Z(Y) \cap \Phi (Y)),Y] = 1$. As $W = [W,Y]$, then $C_W(Z(Y) \cap \Phi(Y)) = 1$, a contradiction.  This proves \fref{eq:2}
 \begin{align*}\label{eq:3}\tag{3}& |[V,x]| = 16 \text{ for all }x \in O_3(X) \setminus \Phi(O_3(X)), \text{ which are inverted}\\& \text{by }t,\text{ or the first case of (ii) holds}.\end{align*}
 Assume $|[V,x_1]| = 4$. Then $W = [V,\langle x_1,x_2\rangle]$ is of order 16 and $y$-invariant. In particular
 \begin{gather*}\label{eq:3.1}|[W,t]| = 4.\tag{3.1}\end{gather*}
  Further  $Y = \langle x_1, x_2 \rangle$ is an elementary abelian group of order 9 in $\GL_4(2)$. Suppose $Y = \langle x_1,x_2 \rangle = O_3(X)$, then the structure of $\GL_4(2)$ shows that $X/O_3(X) \cong \Z_4$, which is the first case in (ii).

 Thus assume  $Y \not= O_3(X)$.  Choose $x_3 \in O_3(X) \setminus \Phi(O_3(X))Y$ with $x_3^t = x_3^{-1}$. By \fref{eq:2} $|[V,x_3]|\leq 16$. By \fref{eq:3.1} $|[V,x_3,t] : [V,x_3,t] \cap W| \leq 2$. This shows that $|W[V,x_3]| \leq 2^6$.
 Set $x_4 = x_3^y$. Then by the same argument $|W\langle [V,x_3], [V,x_4]\rangle| \leq 2^8$. Hence $\langle x_1,x_2,x_3,x_4\rangle$ induces a subgroup of $\GL_8(2)$. In particular $\langle x_1,\ldots ,x_4\rangle$ induces a subgroup of $U = \Z_3 \wr Z_3 \times \Z_3$. Assume that $\langle x_1,\ldots ,x_4\rangle$ is elementary abelian of order $3^4$. As there is a unique elementary abelian subgroup of order $3^4$ in $U$,  there is also some element $x$ of order three in this group with $|[V,x]| = 2^8$ and $x^t = x^{-1}$. But this contradicts $|[V,t]| \leq 8$. Hence we must have that $\langle x_1, \ldots, x_4 \rangle$ does not induce a group with Frattini factor group of order $3^4$. This then shows that $[V,x_3] \leq W$ and the same applies for $[V,x_4]$. This then implies $W = [V,O_3(X)]$ and so $O_3(X)$ cannot act faithfully on $V$. This proves \fref{eq:3}.
\\
\\
Now by \fref{eq:3} we may assume that $|[V,x]| = 16$ for all $x \not\in \Phi(O_3(X))$ with $x^t = x^{-1}$. In particular $[V,x_1] \not= [V,x_2]$, otherwise there is $x \in \langle x_1,x_2 \rangle$, $x^t = x^{-1}$ and $|[V,x]| = 4$, which contradicts \fref{eq:3}.  Set $W = [V, \langle x_1, x_2\rangle]$. This now gives $|W| = 2^6$ and so $Y = \langle x_1, x_2 \rangle \leq \GL_6(2)$ and then $Y$ is  a subgroup of $U = \Z_3 \wr \Z_3$.
If $\langle x_1, x_2 \rangle$ is elementary abelian, then by \fref{eq:3} $|C_{[V,\langle x_1,x_2\rangle]}(x)| = 4$ for any $1 \not= x \in \langle x_1, x_2 \rangle$. But this contradicts coprime action.

Thus $Y = \langle x_1,x_2 \rangle \cong 3^{1+2}$ or $U$. Assume the latter. We have an $\langle y \rangle$-invariant chain $1 <Z(Y) < Z_2(Y) < J(Y) < Y$ all of index three and so $y^2 = t$ must centralize $Y$, a contradiction. Thus
\begin{align*}\label{eq:4}\tag{4}& \text{If }|[V,x]| = 16\text{ for all }x \not\in \Phi(O_3(X)), \text{ then }\\&Y = \langle x_1, x_2 \rangle \text{ is extraspecial of order }27\end{align*}

We have that $|[Y,V,t]| \geq 4$.  In particular as $Z(Y) = Z(U)$ acts fixed point freely on $W = [V,Y]$ we have $[V,Y,t]$ is of order 4. Suppose that $Y \not= O_3(X)$. Then consider again $x_3 \in O_3(X) \setminus Y\Phi(O_3(X))$ with $x_3^t = x_3^{-1}$. Then by \fref{eq:3} also $[x_3,V]$ is of order 16. Consider $W_1 =[V,x_1][V,x_2][V,x_3]$. As $|[W,t]| = 4$ we have that $|W_1|\leq 2^8$. If $[V,x_3] \leq W$, Then $\langle x_1,x_2,x_3 \rangle$ cannot act faithfully on $W$. Thus $|W_1| = 2^8$ and $\langle Y, x_3\rangle \leq \GL_8(2)$ and so contained in a Sylow 3-subgroup $U$ of $\GL_8(2)$ which is isomorphic to $(\Z_3 \wr \Z_3) \times \Z_3$. As $\langle Y, x_3 \rangle$ is not abelian, we see that either $\langle Y, x_3 \rangle = U$ or isomorphic to $3^{1+2} \times \Z_3$. In both cases $t$ inverts some element $x$ of order three in $Z(\langle Y, x_3 \rangle)$ which is not in $Y$. Then we may assume that $x = x_3$. But in $Z(U) = Z(\langle Y, x_3\rangle)$, there are just 3-elements $a$, with $|[W_1,a]| = 2^6$, $b$ with $|[W_1,b]| = 4$ and $a\cdot b$ with $|[W_1,a\cdot b]| = 2^8$. As $|[W_1,x_3]| = 2^4$, this is a contradiction. This contradiction shows $$Y = O_3(X).$$

By \fref{eq:4} $O_3(X)$ is extraspecial of order 27 and exponent 3 and $W = [V,Y]$ is of order $2^6$. As $\Out(O_3(X)) \cong \GL_2(3)$, we see that $X/O_3(X)$ is cyclic of order at most eight or a quaternion group of order eight. Further $|[W,t]| = 4$. These are the remaining cases in (ii).
\end{proof}

\section{Some results about the Sylow 2-subgroup of $\Omega^+_8(2)$}

In this section we will prove some basic facts about $S$, a Sylow 2-subgroup of $\Omega_8^+(2)$. For this we make use of the structure of $\Omega^+_8(2)$ and the fact that $\Omega^+_8(2)$ is a subgroup of $\POmega^+_8(3)$ of odd index. This can be seen in \cite{AschSm}. We will use this paper as a reference for all the properties we need for $\POmega^+_8(3)$. Most of the results of this section also hold for $\POmega^+_8(q)$, $q = 2^{a} > 2$, just by replacing $2$ by $q$. If not, we will comment on it. Further $\POmega_8^+(3)$ has a 2-local geometry (parabolic system) of type
\vspace{-1cm}

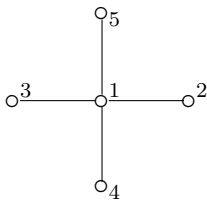
$$\setlength{\unitlength}{2mm}
\begin{picture}(50,40)(-20,-30) \put(0.3,0.3){\makebox(0,0){$\circ^1$}} \put(6.1,0.3){\makebox(0,0){$\circ^2$}}
\put(-5.6,0.3){\makebox(0,0){$\circ^3$}} \put(0.3,5.6){\makebox(0,0){$\circ_5$}}
\put(0.3,-5.9){\makebox(0,0){$\circ_4$}} \put(-0.05,0.4){\line(0,1){5}} \put(-0.05,-0.4){\line(0,-1){5}}
\put(-0.5,0){\line(-1,0){5}} \put(0.4,0){\line(1,0){5}} \end{picture} $$
\vspace{-5cm}

\noindent Here any 4-set containing 1 corresponds to a subgroup $\Omega^+_8(2)$. By \cite{Atlas} $\POmega_8^+(3)$ has an outer automorphism group $\Sigma_4$, which acts on the diagram by permuting $\{2,3,4,5\}$. Further any 3-set containing $1$ corresponds to a subgroup $2^6: \Omega^+_6(2)$ in $\POmega^+_8(3)$.
\\
\\
In case of $\Omega_8^+(2)$, we have the 2-local geometry
$$\Dvier$$
Again any 3-set containing 1 corresponds to a subgroup $2^6 :\Omega_6^+(2)$.
\\
\\
{\bf Notation:} We denote by $Q$ the extra-special subgroup of order $2^9$ and plus type in $S$. This group is $O_2(C_G(R))$, where $G \cong \POmega_8^+(2)$ and $R = Z(S)$ is a root group. We will denote $\bar{S} = S/Q$. If $q > 2$, then $Q$ is a special group.

\begin{lemma}\label{lem:cent} We have $|Z(S)| = 2$, $|Z_2(S)| = 2^2$ and $\bar{S}$ is elementary abelian of order $8$.
\end{lemma}

\begin{proof} This can  easily be calculated inside of $\Omega^+_8(2)$.
\end{proof}

\begin{lemma}\label{lem:64} {\rm(a)} $S$ contains exactly $6$ elementary abelian subgroups $E$ of order $2^6$. Further $|E : E \cap Q| = 2$. They  all are conjugate under $\Aut(S)$. Let $E$ be one of them then $N_{\POmega^+_8(3)}(E)/E \cong \Omega^+_6(2)$ and induces the natural module on $E$.

{\rm (b)} Let $\bar{s}_i$, $i = 1,\ldots, 7$, be the involutions in $\bar{S}$. Then $[Q,s_i]/Z(Q) = C_{Q/Z(Q)}(s_i)$ is of order $16$ for all $i$. There is exactly one $i_0 \in \{1,\ldots , 7\}$ such that $[Q,s_{i_0}]$ is not elementary abelian. In this case $[Q,s_{i_0}] \cong \Z_4 \times \Z_2 \times \Z_2 \times \Z_2$. Furthermore all involutions in the coset $(Q/Z(Q))s_i$ are conjugate. If $s_{i_0}$ is an involution then even all involutions in $Qs_{i_0}$ are conjugate.

{\rm (c)} Let $\bar{s}_i \not= \bar{s}_j$ be as in (b). Then $|[Q/Z(Q),s_i][Q/Z(Q),s_j]| = 64$ and $|[Q/Z(Q),s_i,s_j]| = 4$.
\end{lemma}
\begin{proof}(a)+ (b): By the remark on $\POmega^+_8(3)$ before we have at least $6$ elementary abelian subgroups of order $2^6$ in $S$. Let $E$ be some elementary abelian group of order $2^6$. Assume first $EQ = S$. Then $|E \cap Q| = 2^3$ and so $|C_{Q/Z(Q)}(S)| \geq 2^2$, contradicting \fref{lem:cent}.
\\
\\
Now assume that $|E : E \cap Q| = 4$. Then $|C_{Q/Z(Q)}(E)| \geq 8$. As $|S : EQ| = 2$,  $|C_{Q/Z(Q)}(S)| \geq 4$, contradicting \fref{lem:cent} again.
\\
\\
Thus we have shown that $|E : E \cap Q| = 2$ for each elementary abelian subgroup of order $64$.
\\
\\
We now turn our interest to $\Omega^+_8(2)$. The point stabilizer in the natural $8$-dimensional representation of $\Omega^+_8(2)$ is an extension of an elementary abelian group $E$ of order $2^6$ by $\Omega^+_6(2)$, which induces the natural module on $E$. Then $QE/E$ corresponds to the unique elementary abelian subgroup of order 16 in $S/E$. Hence $[E,Q] = Q \cap E$, which is of index 2 in $E$. This shows that the involutions in $EQ \setminus Q$ are exactly those in $E$. Hence $E$ is the unique elementary abelian subgroup of order $64$ in $EQ$.
\\
\\
As we have 6 conjugates of $E$ under $\Out(\POmega^+_8(3))$, they all have the same normalizer structure.
\\
\\
Now chose notation such that $\bar{s}_i$, $i = 1, \ldots , 7$, are the involutions in $\bar{S}$. Then as there are $6$ conjugates of $E$ under $N_{\Aut(\POmega^+_8(3))}(S)$, we get six $\bar{s}_i$ such that the preimage contains exactly one elementary abelian group of order $64$.

Let $\bar{s}_{i_0}$ be the remaining involution.
Then $\bar{s}_{i_0}$ is fixed by $N_{\Aut(\POmega^+_8(3))}(S)$. We consider $N_{\Aut(\POmega^+_8(3))}(Q)$. We have that $N_{\POmega^+_8(3)}(Q)/Q$ contains an extension of an elementary abelian 3-group of order $3^4$ by $\bar{S}$. This is just the parabolic in $\POmega^+_8(3)$ corresponding to $\{2,3,4,5\}$ in the diagram before. Hence the preimage is  a central product of 4 groups $\SL_2(3)$ (which is just the image in $\POmega_8^+(3)$ or $\SO_4(3) \times \SO_4(3)$) extended by an elementary abelian group of order 8. Further $\Out(\POmega^+_8(3))$ permutes the four groups $\SL_2(3)$ transitively and so $\bar{s}_{i_0}$ must invert $O_3(N_{\POmega^+_8(3)}(Q)/Q)$. This gives that $[s_{i_0},Q]$ is abelian but not elementary abelian as $s_{i_0}$  normalizes some $Q_8$ but does not centralizes this $Q_8$.  This proves most of (b).
The last two statements for $Qs_i$ are obvious. If $s_{i_0}$ can be chosen as an involution, it inverts some element of order 4 in $[Q,s_{i_0}]$ and so all involution in $Qs_{i_0}$ are conjugate. This is (b).\\
\\
(c) By (b) if $|[C_{[Q/Z(Q),s_i]}(s_j)| \geq 8$, then $|C_{[Q/Z(Q),s_i]}(S)| \geq 4$, which contradicts \fref{lem:cent}. Hence $|[Q/Z(Q),s_i,s_j]| = 4$ and then $$|[Q/Z(Q),s_i][Q/Z(Q),s_j]| = 64,$$ which is (c).
\end{proof}

\begin{remark} {\rm The situation in \fref{lem:64} for $q > 2$ is a little bit different. In any case there are exactly $3\cdot q$ many elementary abelian subgroup $E$ of order $q^6$. But under $\Aut(S)$ they form two orbits of length $3$ and $3 \cdot (q-1)$. In any case, also $q=2$, these are the orbits under $N_{Aut(\Omega_8^+(q))}(S)$. In case of $q=2$ there is an additional automorphism of $S$ which connects the two orbits of length 3.}
\end{remark}

In what follows when working in $N_{\POmega^+_8(3)}(E)$, $E \leq S$ elementary abelian of order 64, we often use that $\Omega^+_6(2)$ is isomorphic to $A_8$ and the natural $\Omega^+_6(2)$-module is the irreducible part of the permutation module for $A_8$. For the convenience of the reader we will state  the well known fact of this action, which we will use freely in what follows.
\\
\\
Let $\{v_1, \ldots , v_8\}$ be the basis of the $8$-dimensional permutation module $V$ permuted by $A_8$. Denote by $\bar{}$ the natural homomorphism to $V/\langle v_1+v_2+v_3+v_4+ v_5+ v_6 + v_7 + v_8\rangle$. Then $$E = \langle \overline{v_1+v_2}, \overline{v_1 + v_3}, \overline{v_1 + v_4}, \overline{v_1 + v_5}, \overline{v_1 + v_6}, \overline{v_1 +v_7}, \overline{v_1 + v_8}\rangle.$$ Further we have a Sylow $2$-subgroup of $A_8$
$$\langle (12)(34), (13)(24), (56)(78), (57)(68) \rangle \langle (12)(56), (15)(26)(37)(48) \rangle,$$
which we denote by $S/E$. Then $Q$ is generated by  $$\langle \overline{v_1+v_2}, \overline{v_1 + v_3}, \overline{v_5+v_6}, \overline{v_5+v_7}, \overline{v_1+v_2+v_3+v_4}\rangle$$
and $$\langle (12)(34), (13)(24), (56)(78), (57)(68) \rangle.$$
In the next lemma we will use a certain property of $A_8$, which we now will state
\begin{lemma}\label{lem:A8}{\rm (a)}  Let $x = (12)(34)(56)(78)$ be a $2$-central involution in $A_8$. Then $C_{A_8}(x) = Q_x \langle (1,3,5)(2,4,6), (13)(24) \rangle$ with

\begin{align*}Q_x =& \langle (12)(34)(56)(78), (13)(24)(57)(68),\\& (15)(26)(37)(48), (14)(23)(57)(68), (16)(25)(37)(48)\rangle,\end{align*} extraspecial of order $2^5$ and of $(+)$-type. Besides the involutions in $Q_x$ above there are
involutions like $(12)(34)$, $(12)(56)$ and $(34)(56)$ in $Q_x$.
\\
\\
{\rm (b)} If we consider $A_8$ as $\Omega_6^+(2)$. Then the conjugates of $\overline{v_1+v_2+v_3+v_4}$ are the isotropic vectors. Now $$[x,E] = \langle \overline{v_1+v_2+v_3+v_4}, \overline{v_1 + v_2 + v_5+v_6}\rangle$$ is totally isotropic, but $$[(12)(34),E] = \langle \overline{v_1+v_2}, \overline{v_3+v_4}\rangle$$ is not.
\end{lemma}
\begin{proof} All this can easily be calculated inside of $A_8$.
\end{proof}

\begin{lemma}\label{lem:essential64} {\rm(a)} Let $E$ be one of the elementary abelian subgroups of order $64$ in $S$ as in \fref{lem:64}. Then $Z_2(S) \leq E$.

{\rm (b)} Let $E_1$, $E_2$ be two different elementary abelian subgroups of order $64$ in $S$ as in \fref{lem:64}. Then $[E_1,E_2] = E_1 \cap E_2$ is of order $8$
\end{lemma}

\begin{proof} Set $H = \POmega_8^+(3)$. If $E$ is elementary abelian as in \fref{lem:64} then $N_H(E)/E \cong \Omega^+_6(2)$ and induces the natural module on $E$.

(a) By \fref{lem:cent} $|[E,Z_2(S)]| \leq 2$. But by \fref{lem:A8}(b) there are no transvections in $\Omega^+_6(2)$ and so $[E, Z_2(S)] = 1$. As $C_S(E) = E$, we have $Z_2(S) \leq E$, the assertion.

(b) As $E_1$ and $E_2$ are normal in $S$ we have that $[E_1,E_2] \leq E_1 \cap E_2$. Again let $H = \POmega^+_8(3)$. Then by \fref{lem:64}  $N_H(E_1)/E_1 \cong \Omega^+_6(2)$. We have that $\Omega_6^+(2) \cong A_8$ and it induces on $E_1$ the irreducible part of the permutation module.  The action on $E_1$ shows that $C_S(Z_2(S))/E_1$ is the extra-special group $Q_{E_1}$ of order 32 in $N_S(E_1)/E_1$, see also \fref{lem:A8}. Hence  by (a) $E_1E_2/E_1 \leq Q_{E_1}$ and so $|E_1E_2/E_1| \leq 8$. By \fref{lem:A8}(a) in $N_H(E_1)/E_1$ the group $Q_{E_1}$ is normalized by some group isomorphic to $\Sigma_3$, with an element of order 3 acting fixed point freely on $Q_{E_1}/Z(Q_{E_1})$. In particular this $\Sigma_3$ permutes the elementary abelian subgroups $E$ of order 64 with $EE_1/E_1 \leq Q_{E_1}$.   By \fref{lem:64} $S$ normalizes any elementary abelian group of order 64. Thus the same now is true for $N_{N_H(E_1)/E_1}(Q_{E_1})$. Thus $E_2E_1/E_1$ is normalized by this $\Sigma_3$. As $|E_1E_2/E_1| > 2$, otherwise we would have transvections on $E_2$, which is not possible by \fref{lem:A8}(b) again, we receive that $|E_2E_1/E_1| = 8$. In particular $|E_1 \cap E_2| = 8$, which is the second statement in (b).

There are exactly three elementary abelian subgroup of order 8 in $Q_{E_1}$, which are normal in $N_{N_H(E_1)/E_1}(Q_{E_1})$.
With the notation from \fref{lem:A8}(a) these are
\begin{align*}&\langle (12)(34)(56)(78), (13)(24)(57)(68), (15)(26)(37)(48)\rangle,\\&
\langle (12)(34)(56)(78), (14)(23)(57)(68), (16)(25)(37)(48)\rangle\\&
\text{and }\langle (12)(34)(56)(78), (12)(56), (34)(56)\rangle.\end{align*}
The first two of them are normalized by $\L_3(2)$ in $N_H(E_1)/E_1$ and so if $E_2$ corresponds to one of them we have that $|[E_1,E_2]| = 8$.

Thus assume that $E_2E_1/E_1$ is the third elementary abelian subgroup of $Q_{E_1}$ normalized by $N_{N_H(E_1)/E_1}(Q_{E_1})$. Then besides $Z(Q_{E_1})$ this group just contains non 2-central involutions in $N_H(E_1)/E_1$. As by \fref{lem:A8}(b) $[E_1, Z(Q_{E_1})]$ is totally singular while for a non 2-central involution this is not the case, we receive that $[E_1,E_2] \not= [E_1,Z(Q_{E_1})]$ and so $|[E_1,E_2]| = 8$, the assertion.
\end{proof}

\begin{lemma}\label{lem:Z3} $Z_3(S) \leq Q$ is elementary abelian of order $2^5$ and $$C_S(Z_3(S)) = Z_3(S).$$ Further for any elementary abelian subgroup $E$ as in \fref{lem:64} we have that $|E \cap Z_3(S)| =2^4$
\end{lemma}

\begin{proof} Let $E$ be one of the elementary abelian subgroups of order 64 from \fref{lem:64}. Then by \fref{lem:64} we have that $N_{\POmega_8^+(3)}(E)/E \cong \Omega^+_6(2)$. Hence we can use the notation introduced before \fref{lem:A8}. Then one calculates that $$Z_3(S) = \langle \overline{v_1+v_2}, \overline{v_3 + v_4}, \overline{v_5 + v_6}, \overline{v_1+v_3+v_5+v_7}\rangle  \langle (12)(34)(56)(78) \rangle.$$ Obviously  $Z_3(S)$ is elementary abelian and an easy calculation shows that $C_S(Z_3(S))  = Z_3(S)$. Further for the given $E$  we have that $|E \cap Z_3(S)| =2^4$. As $\Aut(S)$ acts transitively on the elementary abelian subgroups of order 64, this is true for any of them.

Assume there is $x \in Z_3(S) \setminus Q$. Then $|[Q/Z(Q),x]| = 2$, which contradicts \fref{lem:64}. This proves the lemma.
\end{proof}

\begin{lemma}\label{lem:Z3E} Let $x$, $y$ be involutions in $S \setminus Q$, which belong to elementary abelian subgroups $E_x$, $E_y$ of order $64$, $E_x \not= E_y$. Then $E_x \cap E_y = C_Q(x) \cap C_Q(y) \leq Z_3(S)$. If $s_{i_0}$ is an involution, then $C_Q(s_{i_0}) \leq Z_3(S)$.
\end{lemma}

\begin{proof}  By \fref{lem:Z3} $|Z_3(S) : E_x \cap Z_3(S)| = 2 = |Z_3(S) : E_y \cap Z_3(S)|$. Thus $|Z_3 : E_x \cap E_y \cap Z_3(S)| \leq 4$. As $|E_x \cap E_y| = 8$ by \fref{lem:essential64} we get $E_x \cap E_y \leq Z_3(S)$.

We have that $[Z_3(S)/Z(Q),\bar{s}_{i_0}] \leq Z_2(S)/Z(Q)$. Thus $\bar{s}_{i_0}$ centralizes a group of order 16 in $Z_3(S)$ modulo $Z(Q)$. But by \fref{lem:64} any involution which is centralized by $s_{i_0}$ modulo $Z(Q)$ is really centralized. Thus $|C_{Z_3(S)}(s_{i_0})| = 16 = |C_Q(s_{i_0})|$, which gives the second assertion.
\end{proof}

\begin{lemma}\label{lem:frattini} $|S/\Phi(S)| = 16$, $\Phi(S) \leq Q$, $|Q : \Phi(S)| = 2$ and $\Phi(S) = C_Q(Z_2(S))$. Further $S^\prime = \Phi(S)$.
\end{lemma}

\begin{proof} By \fref{lem:cent} $|Z_2(S)| = 4$ and so $\Phi (S) \leq C_S(Z_2(S))$. Further $Z_2(S) \leq Q$ and so $|Q : C_Q(Z_2(S))| = 2$. Thus $Q \not\leq \Phi(S)$. By \fref{lem:cent}  $\bar{S}$ is elementary abelian and so $\Phi(S) \leq Q$.  Let $\bar{S}= \langle \bar{s}_1, \bar{s}_2, \bar{s}_3\rangle$. By \fref{lem:64} we have $|[Q/Z(Q),s_1]| = 16$. By \fref{lem:64} again we have that $|C_{[Q/Z(Q)]}(s_2,s_3)| = 2$. Hence we have $|[[Q,s_1],s_2]| = 4$ and then $[Q/Z(Q),s_1,s_2,s_3] = Z_2(S)/Z(Q)$ of order $2$. As $[Q,s_1]$ is abelian we get by duality that $|Q :[Q,S]| = 2$. Thus $|Q : \Phi(S)| = 2$. As $\Phi(S) \leq C_S(Z_2(S))$, we have that $\Phi(S) = C_Q(Z_2(S))$. This proves the lemma.
\end{proof}

\begin{lemma}\label{lem:elab} Set $\tilde{S} = S/Z(Q)$. If $\tilde{F} \leq \tilde{S}$ is elementary abelian and $|\tilde{F}| \geq 2^6$, then $\tilde{F} \leq \tilde{Q}$.
\end{lemma}

\begin{proof} Suppose $\tilde{F} \not\leq \tilde{Q}$. Then there is some $\bar{s}_i \in \bar{F}$, $\bar{s}_i$ as in \fref{lem:64}. By \fref{lem:64} we have that $\tilde{F} \cap \tilde{Q} \leq C_{\tilde{Q}}(s_i)$ is of order at most 16. Hence $|\bar{F}| \geq 4$ and so there is $\bar{s}_i \not= \bar{s}_j \in \bar{F}$. Now again by \fref{lem:64} $\tilde{F} \cap \tilde{Q} \leq C_{\tilde{Q}}(s_i) \cap C_{\tilde{Q}}(s_j)$, which is of order 4. As $|\bar{F}| \leq 8$, this is a contradiction.
\end{proof}

\begin{lemma}\label{lem:wc}  If $Q_2 \leq S$, $Q_2 \cong Q$, then $Q = Q_2$.
\end{lemma}

\begin{proof} By \fref{lem:cent} we see that $\overline{Q_2}$ is elementary abelian of order at most 8 and so $|Q \cap Q_2| \geq 2^{9-n} \geq 2^6$, where $|\overline{Q_2}| = 2^n$. As $Q$ contains no abelian subgroup of order $2^6$, we have that $(Q \cap Q_2)^\prime = Z(Q) = Z(Q_2)$. In particular $Q_2/Z(Q)$ is an elementary abelian group of order $2^8$ in $S/Z(Q)$, which by \fref{lem:elab} yields $Q_2 \leq Q$ and so $Q = Q_2$.
\end{proof}

\begin{remark} {\rm \fref{lem:wc} is a very strong tool in this paper. It says that whatever the 2-fusion system $\mathcal F$ might be the group $Q$ is weakly $\mathcal F$-closed. This is also true for $q > 2$ and for many groups of Lie type in characteristic two.}
\end{remark}
\begin{lemma}\label{lem:aut} $\Aut(S)$ is a $\{2,3\}$ group with a cyclic Sylow $3$-subgroup of order $3$.
\end{lemma}
\begin{proof} By \fref{lem:wc} $Q$ is characteristic in $S$. Hence $\Aut(S)$ normalizes $Q$. Let $N = C_{\Aut(S)}(\bar{S})$ and $\rho \in N$ be of odd order. By \fref{lem:frattini} $[\rho,Q] \leq \Phi(S)$ and so $[S,\rho] \leq \Phi (S)$, a contradiction.

Thus $N$ is a $2$-group. In particular $\Aut(S)/N$ acts faithfully on $\bar{S}$ and so is a subgroup of $\L_3(2)$. By \fref{lem:64} no element of order 7 can act on $\bar{S}$, so $\Out(S)$ is a $\{2,3\}$-group. By the remark before $\Out(\POmega^+_8(3))$ induces $\Sigma_4$ on $\bar{S}$ and so $\Aut(S)/N \cong \Sigma_4$ and the lemma is proved.
\end{proof}

\begin{lemma}\label{lem:aut1} Let $\rho \in \Aut(S)$ be of order three. Then $[Q/\Phi(S),\rho] = 1$. Further $|[\rho,\bar{S}]| = 4$ and $C_{\bar{S}}(\rho) = \langle \bar{s}_{i_0} \rangle$, $\bar{s}_{i_0}$ as in \fref{lem:64}.
\end{lemma}

\begin{proof} By \fref{lem:wc}  $\rho$ normalizes $Q$. Hence the first assertion follows by \fref{lem:frattini}. As $\bar{s}_{i_0}$ corresponds to the only subgroup of order two in $\bar{S}$, whose preimage contains no elementary abelian subgroup of order 64, we have that $[\rho,\bar{s}_{i_0}] = 1$. As $\rho$ does not centralize $S$, we have that $|[S/\Phi(S),\rho]| = 4$, which are the remaining assertions.
\end{proof}

\section{The essential subgroups of $S$}\label{sec:notnormal}

The aim of this section is to prove that  any essential subgroup of $S$ is normal in $S$. For this we consider the following hypothesis
\begin{hyp}\label{hyp:setup}
\begin{enumerate}
\item[(i)] There is  $E \leq S$, $C_S(E) \leq E$. Set $T = N_S(E)$. Assume there is some group $U$ such that $F^\ast(U) = O_2(U) = E$, $T \in \Syl_2(U)$ and $U_1 = U/O_{2,2^\prime}(U) \cong \PSL_2(2^n)$, $\PSU_3(2^n)$, $n \geq 2$, or $Sz(2^n)$, $n \geq 3$, or cyclic or quaternion. Assume further that $U = \langle T^{U} \rangle$.
    \item[(ii)] $T \not= S$ and $T$ not normal in $U$.
    \end{enumerate}
    \end{hyp}

We will show that no group $U$ satisfies \fref{hyp:setup}. This then will imply that essential subgroups of $S$ are normal in $S$.

In a first step we will show that if $U$ satsfies \fref{hyp:setup}, then $|T: E| =2$. The main toll will be the interplay of $E$, $T$ and $Q$.
\begin{lemma}\label{lem:normal1}
Assume \fref{hyp:setup} Then $Q \not\leq E$ and $E \not\leq Q$.
\end{lemma}

\begin{proof}  As $S/Q$ is abelian and $E$ is not normal in $S$, we have that $Q \not\leq E$. As $C_S(E) \leq E$, we have that $Z(Q) \leq E$.

Assume next that $E \leq Q$. Then elements of $Q \setminus E$ induce transvections to $Z(Q)$ on $E$. This now implies by \fref{lem:representations} that $U_1$ is cyclic or quaternion. In particular $|Q : E| = 2$. By \fref{lem:64} any coset of $Q/Z(Q)$ contains an involution. As $Z(Q) \leq E$ this implies that $U_1$ is of order two. In particular $T = Q$. Now $[T,E/\Phi(E)] = 1$ and so $[\langle T^U \rangle,E/\Phi(E)] = 1$, As $C_U(E) \leq E$, this implies that $T$ is normal in $U$, a contradiction.
\end{proof}

\begin{lemma}\label{lem:normal3} Assume \fref{hyp:setup}. If $|T : E| \geq 4$, then $Z(Q)$ is normal in $U$.
\end{lemma}

\begin{proof} Assume false. As $Z(Q) \leq Z(E)$, we have that $V= \langle Z(Q)^{U} \rangle \leq Z(E)$ is elementary abelian and so by \fref{lem:64} of order at most $2^6$. If it is of order $2^6$, then by \fref{lem:64} $Z(E)$ it is one of the six elementary abelian groups $E_i$. But as $C_S(E_i) = E_i$, we see $E = E_i$ and so $E$ is normal in $S$, which contradicts \fref{hyp:setup}(ii).
\\
\\
Thus $|V| \leq 2^5$. Now  $U/O_{2,2^\prime}(U)$ is involved in $\L_5(2)$. Hence by \fref{lem:representations} and \fref{lem:representation1} we receive that $U$ induces $\L_2(4)$, $\Z_5 \cdot \Z_4$ or $3^2 \cdot \Z_4$ on $V$. In particular
\begin{gather*}\label{eq:T4} |V| \geq 2^4 \text{ and }|T : E| = 4.\tag{1}\end{gather*}
 Furthermore we now see that no element in $T$ can induce a transvection on $V$. Thus as $Z(Q) \leq E$, $Z_2(S) \leq T$ and as there no transvections on $V$ we have
\begin{equation}\label{eq:Z2} Z_2(S) \leq E \tag{2}
\end{equation}
Hence
\begin{equation}\label{eq:Z3Q} Z_3(S) \leq T.\tag{3}\end{equation}

We will show
\begin{gather*}\label{eq:QV} V \not\leq Q.\tag{4}\end{gather*}
Assume $V \leq Q$, then $[N_Q(E),V] \leq Z(Q)$. As no element in $T$ can induce transvections on $V$, we see that
\begin{gather*}\label{eq:QV1} [N_Q(E),V] = 1.\tag{4.1}\end{gather*}
By \fref{eq:Z3Q} $Z_3(S) \leq N_Q(E)$ and so $V \leq C_S(Z_3(S)) = Z_3(S)$ by \fref{lem:Z3}. As $[E,V] = 1$, again by \fref{lem:Z3}   we have $V \not= Z_3(S)$. Now $|V| = 16$. By \fref{lem:normal1} $E \not\leq Q$. Suppose $|E : E \cap Q| \geq 4$. Then by  \fref{lem:64} $|C_Q(E)| \leq 8$. Thus   $|E : E \cap Q| = 2$. Further $|C_Q(V)| = 64$ is a central product of $V$ by some $D_8$. As $|Z_3(S) : V| = 2$, we see that $Z_3(S)$ centralizes $V$ and $E/V$. Hence $Z_3(S) \leq E$ and so
\begin{gather*}\label{eq:QV2} Z_3(S) \leq E \cap Q.\tag{4.2}\end{gather*}
 As $N_Q(E) \leq C_Q(V)$ by \fref{eq:QV1} and $N_Q(E) > E \cap Q$ we now receive that $E \cap Q = Z_3(S)$ and so $2 \leq |E : Z(E)| \leq 4$. which gives $|E^\prime| = 2$, but $E^\prime \leq V$, which is an irreducible $U$-module, a contradiction.
This proves \fref{eq:QV}.
\\
Hence
\begin{equation}\label{eq:Z3notV} Z_3(S) \not\leq V.\tag{5}\end{equation}
We next show
\begin{equation}\label{eq:V2} |VQ/Q| = 2.\tag{6}\end{equation}
If $|VQ/Q| = 8$, then $C_Q(V) = Z_2(S)$, so $E \cap Q \leq Z_2(S)$. Then by \fref{lem:Z3} and \fref{eq:Z3Q} we have $|Z_3(S) : Z_3(S) \cap E| = 8$, a contradiction to $|T : E| =4$ by \fref{eq:T4}.

Assume now $|VQ/Q| = 4$. Then $V = (V \cap Q)\langle x, y \rangle$ with involutions $x$, $y$. By \fref{lem:64} we may choose notation such that $x \in E_x$, $y \in E_y$, $E_x, E_y$ elementary abelian of order 64. By \fref{lem:essential64} $|C_Q(x) \cap C_Q(y)| = 8$. In particular $|E \cap Q| \leq 8$ and so $|Z_3(S) \cap E| \leq 8$. By \fref{lem:Z3} $|C_{Z_3(S)}(x)| = 16$. Hence there is $u \in Z_3(S) \cap C_Q(x)$, which does not centralizes $y$. Then $|V : C_V(u)| = 2$, a contradiction again. This by \fref{eq:QV} proves \fref{eq:V2}.
\\
\\
We now have $V = (V \cap Q)\langle x \rangle$, $x$ an involution, and $E \cap Q \leq C_Q(x)$, which by \fref{lem:64}  is of order at most 32.

Suppose $EQ = VQ$. Assume that there is an elementary abelian subgroup $E_x$  of order 64, with $x \in E_x$. Then $E \leq E_x$. As $C_S(E) \leq E$, we see that $E = E_x$. But then $E$ is normal in $S$, a contradiction.
Thus $x$ corresponds to an involution in $Qs_{i_0}$ in the notation of \fref{lem:64} and so $|C_Q(x)| = 16$ and $C_Q(x) \leq Z_3(S)$ by \fref{lem:Z3E}. But then $Z_3(S)$ induces a transvection on $V$, which is impossible as $Z_3(S) \not\leq E$.

We have $|EQ/Q| \geq 4$. As $|V| \geq 16$, $|V \cap Q| \geq 8$ by \fref{eq:V2} and so by \fref{lem:cent} $|EQ/E| = 4$. This now implies that $|C_Q(E)| \leq 8$ and then $|V| = 16$ and $V \cap Q \leq Z_3(S)$. We have $E \cap Q \leq C_Q(x)$ and so by \fref{lem:Z3} and \fref{eq:QV}, $Z_3(S) \not\leq E$. Thus  there is $u \in Z_3(S) \setminus E$. By \fref{lem:Z3} $[u,V \cap Q] = 1$ and then $u$ induces a transvection on $V$, a contradiction.
This final contradiction proves
the lemma.
\end{proof}

\begin{lemma}\label{lem:normal2} Assume \fref{hyp:setup}. Then $Q \not\leq T$.
\end{lemma}

\begin{proof} Assume that $Q \leq T$. As $E$ is not normal in $S$ and $E \not\leq Q$ by \fref{lem:normal1} we have that
\begin{equation}\label{eq:EQ} 2 \leq |E : E \cap Q| \leq 4.\tag{1}\end{equation}
We first show
\begin{equation}\label{eq:ZQ} Z(Q) \text{ is normal in }U.\tag{2}\end{equation}
Assume false. Then by \fref{lem:normal3} $|T : E| = 2$. As $Q \leq T$ and $Q \not\leq E$ by \fref{lem:normal1} we have $EQ = T$ and $|Q : Q \cap E| = 2$. Let $g \in U$ with $Z(Q)^g \not= Z(Q)$. Hence $Q^g \cap Q \cap E$ is elementary abelian and so $|Q \cap Q^g \cap E| \leq 2^5$. We have $|Q \cap E| = 2^8 = |Q^g \cap E|$.    This implies that $|(Q^g \cap E)(Q \cap E)| \geq 2^{11}$ and so $|(Q^g \cap E)(Q \cap E)/Q \cap E| \geq 8$. Then  $|E: E \cap Q| \geq 8$, which contradicts \fref{eq:EQ}. This proves \fref{eq:ZQ}.
\\
\\
As $[Q,E] \leq E \cap Q$ and $[E \cap Q,Q] \leq Z(Q)$, we see that $E \cap Q$ cannot be normal in $U$. Hence there is $g \in U$ with $E \cap Q \not= (E \cap Q)^g$. As $Z(Q) = Z(Q^g)$ we see that $(E \cap Q^g)/Z(Q)$ is an elementary abelian subgroup of $S/Z(Q)$. By \fref{lem:elab} $|E \cap Q| \leq 2^6$.   Then $|Q/Q \cap E| \geq 2^3$. This implies by \fref{hyp:setup} that $U$ induces $\PSL_2(2^n)$, $Sz(2^n)$ or $\PSU_3(2^n)$, $n \geq 3$, on $E$. By  \fref{lem:representations} in the last case we have that $|E/Z(Q)| \geq 2^{6n} \geq 2^{18}$, which is impossible. In the other cases by the same lemma we see that $|(E/Z(Q))/C_{E/Z(Q)}(x)| \geq 2^n$ for $x \in Q\setminus E$, and $[x,E \cap Q/Z(Q)] = 1$. This now gives $|E : Q \cap E| \geq 2^n \geq 8$, which contradicts \fref{eq:EQ}. This proves the lemma.
\end{proof}

We  fix the following notation:

\begin{notation}\label{not:Q1} {\rm $Q_1 = N_Q(E) = T \cap Q$.  By \fref{lem:normal2} $Q \not\leq T$. Let $Q_2$ be the preimage of $C_{Q/Q_1}(E)$. Then first of all $Q_2 > Q_1$ and so also $Q \cap E < Q_1$.}
\end{notation}



We now prove the first step announced.

\begin{lemma}\label{lem:normal4} Assume \fref{hyp:setup}. Then $|T:E| = 2$.
\end{lemma}

\begin{proof} Assume false. Then by \fref{lem:normal3} $Z(Q)$ is normal in $U$. Set $\tilde{E} = E/Z(Q)$. Then $|\bar{E} : C_{\bar{E}}(Q_1)| \leq 8$ and $Q_1$ acts quadratically on $\bar{E}$. Application of \fref{hyp:setup}  then immediately  \fref{lem:representations} implies $U/O_{2,2^\prime}(U) \cong \PSL_2(2^n)$, $n = 2$ or $3$, or $U$ is solvable. Recall that in case of $\PSU_3(q)$ we would have $|\tilde{E}| \geq 2^{12}$. In the solvable case   by \fref{lem:representation1} $U$ induces $\Z_5 \cdot \Z_4$ or $\Z_3^2 \cdot \Z_4$, $3^{1+2} \cdot \Z_4$, $3^{1+2} \cdot \Z_8$ or $3^{1+2} \cdot Q_8$.  In the solvable cases $|Q_1E/E| = 2$. In all cases  $\tilde{E}$ involves exactly one irreducible non-trivial  $U$-module. Furthermore as $[Z_2(S), E ] \leq Z(Q)$, we have that $[\tilde{E},\widetilde{Z_2(S)}] = 1$ and so $Z_2(S) \leq E$. As $Q_1/E$ cannot induce transvections on $\tilde{E}$ we have
\begin{equation}\label{eq:Z3S} Z_3(S) \leq E \text{ and }|\widetilde{E \cap Q}| \geq 16.\tag{1}\end{equation}
The second assertion in \fref{eq:Z3S} follows from \fref{lem:Z3}.
\\
\\
Assume $U/O_{2,2^\prime}(U) \cong \PSL_2(2^3)$ or $U_1 \cong 3^{1+2}$ acts on $\tilde{E}$. Let $F$ be some section on which $U$ acts. Then by \fref{lem:representations} in the first case and \fref{lem:representation1} in the second case we have that $|F| \geq 2^6$. By \fref{lem:elab} the preimage of $F$ is not abelian. In particular $|E| \geq 2^8$. In case of $\PSL_2(8)$ this gives equality as $|T| \leq 2^{11}$. Then $\tilde{E}$ is extraspecial of order $2^7$. If $U_1 \cong 3^{1+2}$ then either again $\tilde{E}$ is extraspecial or $T/E \cong \Z_4$ and $|C_{\tilde{E}}(U_1)| = 4$. Also in the last case some $U_1$-invariant extraspecial group of order $2^7$ is involved. As  $\widetilde{Q \cap E}$ is an elementary abelian group of index at most 8 in $\tilde{E}$, this shows that this extraspecial group is of $+$-type. But neither $\PSL_2(8)$ nor $3^{1+2}$ is involved in $\Omega_6^+(2)$.
\\
\\
Thus $U/O_{2,2^\prime}(U) \cong \PSL_2(4)$ or $T/E$ is cyclic of order four and $U$ induces $\Z_5 : \Z_4$ or $3^2 : \Z_4$. If $U$ is nonsolvable let $\omega \in U$ be an element of order 5, which is inverted by some element in $Q_1$. Set $W = \langle \omega \rangle$. In the solvable case set $W = O_p(U/E)$, $p = 5$ or $3$. By \fref{lem:representations} and \fref{lem:representation1} there is exactly one non-trivial irreducible module in $\tilde{E}$. Hence $|[\tilde{E},W]/\Phi([\tilde{E},W])| = 16$. By \fref{lem:frattini} $\Phi ([\tilde{E},W]) \leq \tilde{Q}$ and so centralized by $Q_1$ and then by $W$ as well. As $[W, \tilde{E}, W] = [W, \tilde{E}]$ the 3-subgroup lemma implies  $[\tilde{E}, W, C_{\tilde{E}}(W)] = 1$. As some element from $Q_1$ inverts $W$, we see that $|[\tilde{E},W] : [\tilde{E},W] \cap \tilde{Q}| = 4$. By \fref{eq:Z3S}  we have that $|\widetilde{E \cap Q}| \geq 16$.  We have that $C_{\widetilde{Q \cap E}}(W)[W,\tilde{E}]$ contains an elementary abelian subgroup of index 4, which is $C_{\widetilde{Q \cap E}}(W)([W,\tilde{E}] \cap \tilde{Q})$.

Suppose that $C_{\tilde{E}}(W) \cap [\tilde{E},W] = 1$ and so $C_{\widetilde{Q \cap E}}(W)[W,\tilde{E}]$ is elementary abelian of order at least 64 by \fref{eq:Z3S}. But this contradicts \fref{lem:elab} and $E \not\leq Q$ by \fref{lem:normal1}.

Assume now that $C_{\tilde{E}}(W) \cap [W,\tilde{E}] \not= 1$. In particular $\Phi ([W, \tilde{E}) \not= 1$. Then there is some hyperplane  $\tilde{H}$  in  $C_{\widetilde{Q \cap E}}(W)$ such that $X = (C_{\widetilde{Q \cap E}}(W)[W,\tilde{E}])/\tilde{H}$ is extraspecial. Then $U/E$ acts on this group. As this group contains an elementary abelian group of index 4 it is isomorphic to $D_8 \ast D_8$. Thus $U/E$ is isomorphic to a subgroup of $\OO^+_4(2)$ and so  $U/E \cong \Z_3^2 \cdot \Z_4 \leq O^+_4(2)$. Then for the involution $i$ in $T/E$ we have that $[X,i] \cong \Z_4 \times \Z_2$. But $[\bar{E}, i] \leq \bar{Q}$ and so elementary abelian, a contradiction. This proves the lemma.
\end{proof}

According to \fref{lem:normal4} we are now in the situation of \cite[chapter 4]{An}. In particular we may apply \cite[Lemma 4.1(b)]{An} with $E$ in the role of $P_1$ and $T = EQ_1$ in the role of $P$. Further we choose $t \in Q_2 \setminus Q_1$. Then we receive the following:
\begin{notation}\label{not:amalgam}{\rm
\begin{itemize}
\item[(i)] There is a  group $G_1$ with $T \leq G_1$, $E$ normal in $G_1$, $C_{G_1}(E) \leq E$ and $G_1/E \cong D_{2p}$, $p$ odd.
\item[(ii)] There is a group $G_2$ with $T \leq G_2$, $E^t$  normal in $G_2$, $C_{G_2}(E^t) \leq E^t$ and $G_2/E^t \cong D_{2p}$, dihedral of order $2p$ for some odd $p$.
\item[(iii)] $t$ induces an isomorphism $\beta_t$ between $G_1$ and $G_2$, such that $\beta_t$ induces on $T$ the  automorphism induced by $t$.
\item[(iv)] Let $N$ be the largest subgroup contained in $T$, which is normal in both $G_1$ and $G_2$. Then $N^t = N$.
\end{itemize}}
\end{notation}
Our first goal is to prove that $O_p(G_i/N) = 1$ for both $i$.
\begin{lemma}\label{lem:centric} If $O_p(G_1/N) \not= 1$, then $N$ is centric in $S$, in particular $Z(Q) \leq N$.
\end{lemma}

\begin{proof} This is \cite[Lemma 4.2 (b)]{An}.
\end{proof}

\begin{lemma}\label{lem:P1} Assume $O_p(G_1/N) \not= 1$. Then $G_1/N \cong \Z_2 \times D_{2p}$.
\end{lemma}

\begin{proof} We have that $|T : O_2(G_1)  \cap O_2(G_2)| = 4$. As $O_p(G_i/N)$ centralizes $O_2(G_i/N)$ we have that  $(O_2(G_1)  \cap O_2(G_2))/N$ is normal in $G_i/N$ for both $i$, and so we have that $N = O_2(G_1)  \cap O_2(G_2)$. Thus $|T/N| = 4$. As $Q_1 \not\leq E$, there is some involution in $G_1/N \setminus O_2(G_1)/N$. This is the assertion.
\end{proof}

We now fix the following hypothesis and notation.
\begin{hyp}\label{hyp:nonstrict} We adopt the notation $G_1$, $G_2$, $N$ and $t$ from \fref{not:amalgam} and continue with \fref{not:Q1}
\begin{itemize}
\item[(i)] $G_i/N \cong Z_2 \times D_{2p}$.
\item[(ii)] Let $E = N\langle u \rangle$ and let $u_1 \in Q_1 \setminus E$.
\item[(iii)] We choose notation as in \fref{not:amalgam}(iii) above such that $(uN)^t = (uu_1)N$.
\end{itemize}
\end{hyp}

\begin{lemma} \label{lem:NnotQ} Assume \fref{hyp:nonstrict}. Then $N \not\leq Q$. In particular $2 \leq |N : N \cap Q| \leq 4$ and $E \cap Q \leq N$, i.e. $E \cap Q = N \cap Q$. Furthermore $[u,N] \leq N \cap Q$.
\end{lemma}

\begin{proof} Suppose false. Then $N \leq Q$ and so $\Phi(N) \leq Z(Q)$. If $Z(Q) = \Phi(N)$, then $u_1$ acts trivially on $N/\Phi(N)$ and the same applies for $O_p(G_1/N)$. By  \fref{lem:P1} then $O_p(G_1/N)$ centralizes $E$, a contradiction. Thus we have
\begin{center} $N$ is elementary abelian\end{center}
As $E = N\langle u \rangle = (E \cap Q)\langle u \rangle$ and $Z(Q) \leq N$, we have that the preimage $Q_3$ of $C_{Q/Z(Q)}(u)$ is contained in $Q_1$. Thus $|Q_3 : Q_3 \cap N| \leq 2$. Furthermore as $[Q/Z(Q),u] = C_{Q/Z(Q)}(u)$ by \fref{lem:64}(b), we see $[E,Q] \leq Q_1$ and so $Q_2 = Q$.
Now $Q$ normalizes $EQ_1 = T$.  By \fref{lem:centric}  $Z(Q) \leq N$. Thus  $Q$ also normalizes $N$ and then it normalizes $T/N$, a fours group. Now $|Q : N_Q(E)| \leq 2$ and as $|Q_1 : N| = 2$, we get $|Q : N| \leq 4$. But $Q$ contains no elementary abelian subgroup of index 4, a contradiction.  This proves the first assertion.
\\
\\
Assume that $E = N(E \cap Q)$.  As $[Q,t] \leq Z(Q) \leq N$, we have that $E^t = E$, a contradiction.
Thus $E \cap Q = N \cap Q$. As $|E : E \cap Q| \leq 8$ and $N \not= E$, we see  $|N : N \cap Q| \leq 4$, which is the second assertion.
The third one follows as $EQ/Q$ is elementary abelian and so $E^\prime \leq Q$.
\end{proof}

Let $H_1$ be the free amalgamated  product over $T$  of $G_1$ and $G_2$ and $H = H_1 \langle \beta_t \rangle$. Then $H$ acts on $N$ and we denote by $\tilde{H} = H/C_H(N)$. In fact $\tilde{H_1}/(N/Z(N))$ is a faithful completion.

\begin{lemma}\label{lem:faithful}  Assume \fref{hyp:nonstrict}. Set $ N_1 = \langle Z(Q)^{H} \rangle$ and $\hat{N} = N/N_1\Phi(N)$. Then $C_{G_i}(\hat{N})/N \geq O_p(G_i/N)$, $i = 1,2$.
\end{lemma}

\begin{proof} Assume false. We just have to prove the lemma for $i = 1$. The assertion for $i=2$ then follows by application of $\beta_t$.

We have that $H$ acts on $\hat{N}$. By \fref{lem:centric} $Z(Q) \leq N$ and so $N_1 \leq Z(N)$. By \fref{lem:NnotQ} we have that
\begin{align*}\label{eq:centu1}\tag{1} |\hat{N} : C_{\hat{N}}(u_1)| \leq 4 \end{align*}
and so $p \leq 5$. Furthermore
$[u,\hat{N}] \leq \widehat{Q \cap N}$ by \fref{lem:NnotQ}. Thus
\begin{align*}\label{eq:quadratic}\tag{2} \langle u, u_1 \rangle \text{ acts quadratically on }\hat{N}.\end{align*}
This implies
\begin{align*}\label{eq:trivial}\tag{3} [[O_p(G_1/N),\hat{N}],u] = 1. \end{align*}
If $[\hat{N},u] = 1$, then as $E = N\langle u \rangle$, $[\hat{N},E] = 1$. As $[O_p(G_1/N),\hat{N}] \not= 1$ we receive that $E = C_T(\hat{N})$. Hence $E^t = E$, a contradiction. Thus
\begin{gather*}\label{eq:nontrivial} [\hat{N},u] \not= 1.\tag{4}\end{gather*}
\\
Assume that $[u_1, C_{\hat{N}}(O_p(G_1/N))] = 1$. Then by \fref{eq:trivial} $$|[\hat{N},u]| \cdot |[u_1, [O_p(G_1/N),\hat{N}]]| = |[\hat{N},uu_1]|.$$ But by \fref{hyp:nonstrict}(iii) $uN$ is conjugate to $uu_1N$ via $t$, a contradiction. Thus
\begin{align*}\label{eq:OP1notcent}\tag{5}  [u_1,C_{\hat{N}}(O_p(G_1/N))] \not= 1.\end{align*}
As $\beta_t(u_1) = u_1$, $\langle t, u_1\rangle \leq Q$ and $\beta_t(O_p(G_1/N)) = O_p(G_2/N)$ also $[u_1,C_{\hat{N}}(O_p(G_2/N))]\not= 1$. This then implies by \fref{eq:centu1} and \fref{eq:OP1notcent} that
\begin{align*}\label{eq:GL4}\tag{6} |[u_1,\hat{N}]| = 4 \text{ and } |[\hat{N}, O_3(G_i/N)]| = 4, i = 1,2.\end{align*}
In particular
\begin{gather*}\label{eq:p=3} p = 3.\tag{7}\end{gather*}
 As $t \in Q$, we also see that $\langle u_1, t \rangle$ acts quadratically on $\hat{N}$. And so $[[\hat{N}, O_3(G_i/N)],u_1,t] = 1$. Thus $[\hat{N}, O_3(G_1/N)] \cap [\hat{N}, O_3(G_2/N)] \not= 1$. In particular by \fref{eq:GL4}
\begin{gather*} |[\langle O_3(G_1/N), O_3(G_2/N)\rangle, \hat{N}]| \leq 8. \tag{8}\end{gather*}
But $\GL_3(2)$ does not have a subgroup $\Z_2 \times \Sigma_3$, a contradiction. This proves the lemma.
\end{proof}

\begin{lemma}\label{lem:ZN} Assume \fref{hyp:nonstrict}. Let $N_1$ be as in \fref{lem:faithful}, then $ N_1  \not\leq Q$.
\end{lemma}

\begin{proof} Assume false, i.e. $N_1 \leq Q$. Recall that $N_1 \leq Z(N)$. We first show
\begin{align*}\label{eq:ZNQ}\tag{1} C_{G_1}(N_1)/N \geq O_p(G_1/N)].
\end{align*}

Suppose false. As $N_1 \leq Q$ by assumption, we have that $Q_1$ induces transvections on $N_1$. In particular $p = 3$ and $$|[O_3(G_1/N), N_1]| = 4.$$ As $\beta_t$ moves $G_1$ to $G_2$ in $H$, we also have that $$|[O_3(G_2/N), N_1]| = 4.$$
As $[[O_3(G_i/N),u_1] \not= 1$, $i = 1,2$ we have
$$Z(Q) \leq [O_3(G_1/N), N_1] \cap [O_3(G_2/N), N_1].$$
Then $$|[\langle O_3(G_1/N), O_3(G_2/N) \rangle, N_1]| = 8.$$ As $T/N$ normalizes $O_3(G_1/N)$ and $O_3(G_2/N)$ as well and $u^t \not= u$, we also see that $[u,N_1] \not= 1$. This implies that $G_1/N \cong Z_2 \times \Sigma_3$ is isomorphic to a subgroup of $\GL_3(2)$, which is absurd. This proves \fref{eq:ZNQ}.
\\
\\
As $Z(Q) \leq N_1$ we now have that $Z(Q)$ is normal in $G_1$ and $G_2$ as well and so $Z(Q) = N_1$. By \fref{lem:faithful} we have $[O_p(G_1/N),N] \leq Z(Q)\Phi(N)$ and so even $[O_p(G_1/N),N] \leq \Phi(N)$, which is absurd as $O_p(G_1)$ would act  trivially on $N$
and then also on $E$. But $C_{G_1}(E) \leq E$ by \fref{not:amalgam}. This proves the lemma.
\end{proof}

\begin{lemma}\label{lem:N12} Assume \fref{hyp:nonstrict}. Then $$|\Omega_1(Z(N))  : \Omega_1(Z(N)) \cap Q| = 2.$$
\end{lemma}

\begin{proof} Assume $|\Omega_1(Z(N))  : \Omega_1(Z(N)) \cap Q| \geq 4$. By \fref{lem:ZN} $\Omega_1(Z(N)) \not\leq Q$.  Hence there are $x, y \in \Omega_1(Z(N)) \setminus Q$, $y \not\in xQ$. By \fref{lem:64} we may choose notation such that $x \in E_x$, $y \in E_y$, $E_x$, $E_y$ elementary abelian of order 64. By \fref{lem:Z3E} $N \cap Q \leq E_x \cap E_y$ and so $|N \cap Q| \leq 8$ by \fref{lem:essential64}. This by \fref{lem:NnotQ}  implies that $|N| \leq 2^5$ and $N = Z(N)$. As $N$ is centric by \fref{lem:centric} $E_x \cap E_y  \leq N$. Thus  $|N| = 2^5$ and in particular $Z_2(S) \leq N \leq E$. Then $Z_3(S) \leq N_S(E)$. By \fref{lem:NnotQ}  we have $N \cap Q = E \cap Q$. Thus by \fref{lem:Z3} $|Q_1E/E| \geq 4$, a contradiction.   This proves the lemma.
\end{proof}
By \fref{lem:ZN} there is $x \in N_1 \setminus Q$. As $x \in \Omega_1(Z(N))$ we now by \fref{lem:N12} get $\Omega_1(Z(N)) = (\Omega_1(Z(N)) \cap Q)\langle x \rangle$. Further $N \cap Q \leq C_Q(x)$. By \fref{lem:essential64} we have that
\begin{itemize}
\item $|N \cap Q| \leq 2^5$ in case that $x \in E_x$, $E_x$ elementary abelian of order 64. Further  $|N| \leq 2^7$; or
\item $|N \cap Q| \leq 2^4$ and $|N| \leq 2^6$.
\end{itemize}

\begin{lemma}\label{lem:N1struc} Assume \fref{hyp:nonstrict}. Then $N \not= C_Q(x)\langle x \rangle$.
\end{lemma}

\begin{proof} Assume false. Then by \fref{hyp:nonstrict}(ii)  $$E = N\langle u \rangle = C_Q(x)\langle x, u \rangle.$$ Suppose first that $x
\in E_x$, $E_x$ elementary abelian of order 64. Then by \fref{lem:64} $E_x = C_Q(x)\langle x \rangle \leq E$.  As stated in \fref{lem:64}(a)  $S/E_x$ is isomorphic to a Sylow 2-subgroup of $\Omega_6(2) \cong A_8$ and so $E/E_x$ corresponds to some involution in $S/E_x$. In $A_8$ we see that any involution in $S/E_x$ is centralized by some  subgroup of order at least 8. This follows from the fact that there is a normal elementary abelian subgroup of order 16 in $S/E_x$. This implies that $|N_S(E)/E| \geq 4$, a contradiction.

Let $x$ correspond to $s_{i_0}$. Then $|[Q,x] : C_{[Q,x]}(x)| = 2$ by \fref{lem:64}. So $[Q,x] \leq Q_1$. By \fref{lem:Z3E} we have that $C_Q(x) \leq Z_3(S)$ and of index 2 in $Z_3(S)$. Hence also $Z_3(S) \leq Q_1$, which now implies $$|Q_1/C_Q(x)| = 4 = |Q_1 : Q_1 \cap N|,$$ a contradiction again.
\end{proof}

\begin{lemma}\label{lem:N4} Assume \fref{hyp:nonstrict}. Then $|N : N \cap Q| = 4$.
\end{lemma}

\begin{proof} By \fref{lem:NnotQ} we otherwise may assume that $N = (N \cap Q)\langle x \rangle$, $x \in \Omega_1(Z(N))$. Now $N \leq C_Q(x)\langle x \rangle$. By \fref{lem:64} $C_Q(x) \langle x \rangle$ is abelian. As $N$ is centric we have $N = C_Q(x)\langle x \rangle$, which contradicts \fref{lem:N1struc}.
\end{proof}

\begin{lemma}\label{lem:L32} Both $G_1$ and $G_2$ act faithfully on $N_1$. Further $$N_1/N_1 \cap \Phi(N)| \geq 2^4.$$
\end{lemma}
\begin{proof} Assume that $O_p(G_1/N)$ centralizes $N_1/N_1 \cap \Phi(N)$. Then by \fref{lem:faithful} it centralizes $N/\Phi(N)$ and so also $N$. As $|E/N| = 2$ it centralizes  $E$ a contradiction to $C_{G_1}(E) \leq E$ by \fref{not:amalgam}.

 Suppose that $N_1/\Phi(N) \cap N_1$ is centralized by some non-trivial element in $G_1/N$, then this is in $O_2(G_1/N)$. But then it inverts $O_p(G_2/N)$ and so $[N_1/N_1 \cap \Phi(N),O_p(G_2/N)] = 1$. Application of $\beta_t$ shows in both cases that  $N_1 = \langle Z(Q)^T \rangle = Z(Q)$, a contradiction to \fref{lem:ZN}.

As $\GL_3(2)$ does not contain a subgroup isomorphic to $\Z_2 \times D_{2p}$, we see that $|N_1/N_1 \cap \Phi(N)| \geq 2^4$.
\end{proof}

\begin{lemma}\label{lem:Nnonabelian} $N^\prime \not= 1$.
\end{lemma}

\begin{proof} Assume $N = Z(N)$. By \fref{lem:N4} $|Z(N) : Z(N) \cap Q| = 4$. By \fref{lem:N12}  $|\Omega_1(Z(N)) : \Omega_1(Z(N)) \cap Q| = 2$.  Thus $N/N \cap Q$ is cyclic of order 4, which contradicts $S/Q$ elementary abelian.
\end{proof}

\begin{lemma}\label{lem:nonnonstrict} \fref{hyp:nonstrict} does not hold.
\end{lemma}

\begin{proof} Assume that \fref{hyp:nonstrict} holds. By \fref{lem:Nnonabelian}  $Z(N) \not= N$. Further by \fref{lem:N4} $|N : N \cap Q| = 4$. According to \fref{lem:N12} let  $x \in \Omega_1(Z(N)) \setminus Q$. Then $N \cap Q \leq C_Q(x)$ and so $$N = (N \cap C_Q(x))\langle x,y \rangle \text{ with }y^2 \in Q.$$ We have
$N_1 \cap Q \leq C_Q(x)  \cap C_Q(y)$. By \fref{lem:L32} $|N_1/N_1 \cap \Phi(N)| \geq 2^4$. Hence $|Z(N)/Z(N) \cap \Phi(N)| \geq 2^4$ and so by \fref{lem:64}
$C_Q(x) \cap C_Q(y) \leq Z(N)$. By \fref{lem:N12} $|Z(N) : Z(N) \cap Q| = 2$ and then $N_1 = Z(N)$ is of order  $2^4$. By \fref{lem:L32} $\Phi(N) \cap Z(N) = 1$, which yields $\Phi(N) = 1$, a contradiction. This proves the lemma.
\end{proof}

\begin{proposition}\label{prop:strict} Assume \fref{hyp:setup}. Then $G_1/E \cong D_{2p}$  and $G_1/N$ is strictly $2$-constrained.
\end{proposition}

\begin{proof} By \fref{lem:normal4} we have that $G_1/E \cong D_{2p}$. By \fref{lem:nonnonstrict} $O_p(G_1/N) = 1$.  The assertion now follows by \cite[Lemma 4.2 (c)]{An}.
\end{proof}

\begin{lemma}\label{lem:Gold} Assume \fref{hyp:setup}. Then one of the following holds
\begin{itemize}
\item[(i)] $G_1/N \cong G_2/N \cong \Sigma_4$; or
\item[(ii)] $G_1/N \cong G_2/N \cong \Z_2 \times \Sigma_4$.
\end{itemize}
\end{lemma}

\begin{proof} We consider the amalgam $(G_1/N, G_2/N)$. As $G_1/E \cong D_{2p} \cong G_2/E^t$, where $EE^t = T$, we may apply \cite[Theorem 1]{Fan}. Recall that $G_1/N \cong G_2/N$. By \fref{prop:strict} we either have a Goldschmidt amalgam or an amalgam of type $^{2}F_4(2)$ or $^{2}F_4(2)^\prime$. In the latter the two  groups $G_1/N$ and $G_2/N$ would  not be isomorphic. Thus we have one of the 15 Goldschmidt amalgams. Application of \cite{Gold} yields that there are only two pairs, which consists of isomorphic constrained groups, the ones from the statement.
\end{proof}

\begin{proposition}\label{prop:essential} \fref{hyp:setup} does not hold. In particular any essential subgroup is normal in $S$.
\end{proposition}

\begin{proof} Assume \fref{hyp:setup} holds. Then we have the situation of \fref{lem:Gold}. As $|T : E| = 2$ we have that $E/N = O_2(G_1/N)$. Then $E^t/N$ is the other maximal elementary abelian subgroup of $T/N$. As $t \in Q$, we have $[E,t]N/N \leq Q_1N/N$.  As $E/N \cup E^t/N = \Omega_1(T/N)$, we have that $Q_1N/N$ contains some element $v$ of order 4. In particular
\begin{gather*}\label{eq:1} 1\not= Z(Q)N/N \leq T/N.\tag{1}\end{gather*}
  We have some element $u \in E \setminus Q_1$ such that $(uN)^t = (uvN)$.  Now $u$ inverts some element $v$ of order 4 in $[Q,u]$ with $v^2 \in Z(Q)$.  By \fref{lem:64} the coset $uQ$ is uniquely determined. In the language of that lemma $uQ = \bar{s}_{i_0}$.
Suppose $G_1/N \cong \Z_2 \times \Sigma_4$. Then for any $w$  in the preimage of $Z(G_1/N)$ also $(wuN)$ inverts some element of order four in $[Q,wu]$. As there is a unique coset with this property we have $wN \in Q_1N/N$. Hence in either case
\begin{gather*}\label{eq:2}Q_1N/N\text{  centralizes }vN,\tag{2}\end{gather*}
 i.e. $Q_1N/N = \langle v, Z(T/N) \rangle$. Thus
\begin{gather*}\label{eq:Q1} \Omega_1(Q_1N/N) = (E \cap E^t)/N.\tag{3}\end{gather*}

As $C_S(E) \leq E$, $Z(Q) \leq E$, we have that $Z_2(S) \in Q_1$ and so $[v,Z_2(S)] \leq N$. Hence by \fref{eq:Q1} $Z_2(S) \leq E$. Now $Z_3(S) \leq Q_1$. As $v \in Q$ we have by \fref{eq:2}  $[v,Z_3(S)] \leq N \cap Z(Q) = 1$. But this contradicts \fref{lem:Z3}. This proves the first part of the proposition.

Now assume that $E$ is some essential subgroup of $\mathcal F$, which is not normal in $S$. Then by the remark after \fref{lem:strongembedded} $E$ satisfies \fref{hyp:setup}(i). If $E$  would not be normal in $S$ it also satisfies \fref{hyp:setup}(ii). Hence this implies that all essential subgroups must be normal in $E$.
\end{proof}

\section{ Essential subgroups of $S$ normal in $S$}\label{sec:essnormal}

In this section we will determine which subgroups $E$ normal in $S$ could be essential and how does their normalizer looks like. Recall that by \fref{prop:essential} these then will be all the potential essential subgroups. This is by far not as difficult as the proofs in \fref{sec:notnormal} as we know the exact structure of $T = N_S(E) =S$.

\begin{lemma}\label{lem:noA5} Let $U$ be some group, $F^\ast(U) = O_2(U)$, $U/O_{2,2^\prime}(U) \cong \PSL_2(q)$, $\PSU_3(q)$ or $Sz(q)$, $q = 2^n$,  $n \geq 2$.   Then a Sylow $2$-subgroup of $U$ is not isomorphic to $S$.
\end{lemma}
\begin{proof} Assume false, i.e. $S$ is a Sylow 2-subgroup of $U$. Then $N_U(S)/S$ involves a cyclic group $\Z_{q-1}$. In case of $\PSU_3(q)$
it even involves a cyclic group of order $(q^2-1)/\gcd (3,q+1)$.  By \fref{lem:aut} we receive that $U/O_{2,2^\prime}(U) \cong \L_2(4)$. As $|Q/\Phi(S)| = 2$ by \fref{lem:frattini} we receive that $Q \leq O_2(U)$. Now $Q$ is normal in $U$  by \fref{lem:wc} and so $(U/Q)/O(U/Q)  \cong \Z_2 \times \L_2(4)$. Now by \fref{lem:aut} an element of order three in $N_U(S)$ induces two orbits of length three on the elementary abelian groups of order $64$ in $S$. Thus $O_2(U) = \langle Q,s_{i_0}\rangle$ in the notation of \fref{lem:64}. By \fref{lem:64} $[Q,s_{i_0}] \cong \Z_4 \times \Z_2\times \Z_2 \times \Z_2$. As $\L_2(4)$ cannot act non trivially on such a group it shows that $C_U([Q,s_{i_0}])$ covers $U/O_{2,2^\prime}(U) \cong \L_2(4)$. But then also $C_U(Q)$ covers $U/O_{2,2^\prime}(U)$, a contradiction as $C_S(Q) =Z(Q)$.
\end{proof}

\fref{lem:noA5} together with \fref{lem:strongembedded} shows that for an essential subgroup $E$, which is normal in $S$  that $S/E$ is cyclic or quaternion. We will investigate these situations in the following lemma and show that $|S/E| = 2$.

\begin{lemma}\label{lem:essentialnormal1} Let $U$ be a group with Sylow $2$-subgroup isomorphic to $S$. Assume that $F^\ast(U) = O_2(U)$,  $S/O_2(U)$ is cyclic or quaternion and $U/O_2(U) = O(U/O_2(U))(S/O_2(U))$, $[S/O_2(U),O(U/O_2(U))] = O(U/O_2(U))$, then $|S :O_2(U)| = 2$.
\end{lemma}

\begin{proof} Assume  that $|S/O_2(U)| > 2$. By \fref{lem:frattini} $S/S^\prime$ is elementary abelian, so $S/O_2(U)$ must be quaternion. As $S/Q$ is elementary abelian, we see that $Q \not\leq O_2(U)$ and $S/O_2(U)$ is a quaternion group of order eight. Further $Z(Q) \leq O_2(U)$ and so $QO_2(U)/O_2(U)$ is elementary abelian. Thus  $|QO_2(U)/O_2(U)| = 2$ and $|S : QO_2(U)| = 4$. This shows $|\overline{O_2(U)}| =2$. By \fref{lem:64} there is some elementary abelian group $E$ of order 64 with $E \not\leq O_2(U)Q$. But this contradicts the structure of a quaternion group.
\end{proof}

\begin{lemma}\label{lem:essentialnormal2} Let $U$ be as in \fref{lem:essentialnormal1} then $U/O_2(U) \cong \Sigma_3$. Further
\begin{itemize}
\item[(1)] If $Q \not\leq O_2(U)$, then $Z_2(S)$ is normal in $U$ and $C_U(Z_2(S)) = O_2(U)$. Furthermore $U$ normalizes each of the six elementary abelian subgroups $E$ from \fref{lem:64}.
    \item[(2)] If $Q \leq O_2(U)$ then $U$ normalizes exactly three of the elementary abelian subgroups of order $64$ in $S$ as described in \fref{lem:64} and $\bar{s}_{i_0} \not\in \overline{O_2(U)}$. We have $\overline{O_2(U)} = \langle \bar{s}_i, \bar{s}_j\rangle$, where $\bar{s}_i$, $\bar{s}_j$ and $\bar{s}_i\bar{s}_j$ all correspond to elementary abelian groups of order $64$. Further $[O_2(U),U] \leq Q$.
\end{itemize}
\end{lemma}

\begin{proof}  By \fref{lem:essentialnormal1} $|S/O_2(U)| = 2$. Let first $Q \not\leq O_2(U)$. Then $S = O_2(U)Q$. Let $x \in Q \setminus O_2(U)$, $x^2 = 1$, this exists as $Q = \Omega_1(Q)$, and $\rho \in U$ of odd order with $\rho^x = \rho^{-1}$. Assume that $[Z(Q),\rho] =1$. Then $(Q \cap O_2(U))^\rho/Z(Q)$ is an elementary abelian subgroup of $S/Z(Q)$ of order $2^7$. By \fref{lem:elab} $Q \cap O_2(U)$ is normalized by $\rho$. As $x$ centralizes $(Q \cap O_2(U))/Z(Q)$ also $\rho$ centralizes this group. Further $[O_2(U),x] \leq Q \cap O_2(U)$ and so also $[\rho, O_2(U)] \leq O_2(U) \cap Q$, which gives the contradiction $[\rho,O_2(U)] = 1$.

Hence $Z(Q) \not= Z(Q)^\rho \leq Z(O_2(U))$. This shows $Z(O_2(U)) = Z_2(S)$. Now $U$ induces $\Sigma_3$ on $Z_2(S)$. Further $[C_U(Z_2(S)),x] \leq O_2(U)$. This shows that $C_U(Z_2(S)) \leq N_U(S)$. By \fref{lem:aut} this is of order at most three and so by assumption $U/O_2(U) \cong \Sigma_3$ and $Z_2(S)$ is normal in $U$, which is (1).
\\
\\
Let now $Q \leq O_2(U)$. By \fref{lem:wc} $Q$ is normal in $U$. Then as $\bar{S}$ is elementary abelian, we have that $\overline{O_2(U)}$ is centralized by $\bar{U}$.

Let first $\overline{O_2(U)} = \langle \bar{s}_1,\bar{s}_2 \rangle$ in the notation of \fref{lem:64}, where $\bar{s}_i$ correspond to elementary abelian subgroups of order 64 in $S$. Then $O(U/O_2(U))$ acts on $Q_1 = [Q/Z(Q),s_1,s_2]$, which by \fref{lem:64} is of order 4. Thus $U_1 = C_{O(U/O_2(U))}(Q_1)$ is of index at most three in $O(U/O_2(U))$. But $U_1$ centralizes $\bar{s}_2$ and so it centralizes also $[s_1,Q]$. This now implies that $U_1$ centralizes the unique elementary abelian subgroup of order 64 in $Q\langle s_1\rangle$. By the $A \times B$-Lemma it then also centralizes $O_2(U)$ and so $U_1 = 1$. Hence again $U/O_2(U) \cong \Sigma_3$.
\\
\\
Suppose now that $\bar{s}_{i_0} \in \overline{O_2(U)}$. Then $C_{[Q/Z(Q),s_{i_0}]}(O(U/O_2(U)) \not= 1$. As $|[Q/Z(Q),s_1,s_{i_0}]| = 4 = |[Q/[Q,s_1],s_{i_0}]|$ we receive that $$[[Q/Z(Q),s_{i_0}],O(U/O_2(U))] = 1.$$ Hence again by the $A \times B$-Lemma we get a contradiction. Thus $\bar{s}_{i_0} \not\in \overline{O_2(U)}$ and so the three non-trivial cosets of $Q$ in $O_2(U)$ correspond to three elementary abelian subgroups of order 64, which all are normalized by $U$. This is (2).
\end{proof}

We now collect the results of this section and apply  to the case of essential subgroups, which are normal in $S$. For this we introduce the following groups:
\\
Let $F_2, \ldots, F_5$ be the four subgroups of index two in $S$ containing $Q$ such that $\bar{F}_i$ is a fours group, which does not contain $\bar{s}_{i_0}$. Recall that there are exactly four such subgroups in $\bar{S}$ by \fref{lem:64}. Set $F_1 = C_S(Z_2(S))$.
\begin{proposition}\label{prop:normalinS} Let $\mathcal F$ be a saturated fusion system on $S$. Let $F$ be some essential subgroup in $S$ such that $F$ is normal in $S$. Then $F \in \{F_1, \ldots , F_5\}$. Further
\begin{itemize}
\item[(1)] If $N_{\mathcal{F}}(S) = S$, then $N_{\mathcal{F}}(F)/F \cong \Sigma_3$.
\item[(2)] If $|N_{\mathcal{F}}(S)| $ is divisible by $3$, then in case of $F = F_1$, we have that $N_{\mathcal F}(F)/F \cong \Sigma_3 \times \Z_3$. If $F$ is one of $F_2, \ldots F_5$ and $ N_{\mathcal F}(S) \leq N_{\mathcal F}(F)$, then $N_{\mathcal F}(F)/F \cong \Sigma_3 \times \Z_3$. In this case either $F$ is the only essential subgroup in $\{F_2, \ldots, F_5\}$ or all groups $F_2, \ldots F_5$ are essential, but $F$ is the only one, which is normalized by $N_{\mathcal F}(S)$.
\end{itemize}
\end{proposition}

\begin{proof} We apply \fref{lem:strongembedded}. By \fref{lem:noA5} we have that \fref{lem:strongembedded}(i) holds; i.e.  $S/F$ is cyclic or quaternion. By \fref{lem:essentialnormal1} $|S: F| = 2$. Thus $N_{\mathcal F}(F)/F =K(S/F)$, where $K$ is of odd order. By \fref{lem:essentialnormal2} $[K,S/F] \cong \Z_3$ and $F \in \{F_1, \ldots ,F_5\}$  and so $N_{\mathcal F}(F)/F \cong \Sigma_3 \times L$, where $L = 1$ or $\Z_3$ depending on $N_{N_{\mathcal F}(S)}(F)$. Recall that by \fref{lem:aut} $N_{\mathcal F}(S)/S \leq \Z_3$.  Hence we have (1).

Assume now that $N_{\mathcal F}(S)/S \cong \Z_3$. As $N_{\mathcal F}(S)$ normalizes $Z_2(S)$ and then also $C_S(Z_2(F)) = F_1$, we have that $L \cong \Z_3$ in case of $F = F_1$.

There are exactly four fours groups in $\bar{S}$, which do not contain $\bar{s}_{i_0}$. By definition these are exactly the $\bar{F}_i$, $i = 2,\ldots , 5$. Hence $N_{\mathcal F}(S)$ normalizes one of the $F_i$ and acts transitively on the remaining. Suppose that $F \in \{F_2, \ldots, F_5\}$ and $N_{\mathcal F}(S)$ normalizes $F$, then we have $L \cong \Z_3$ again.  As $N_{\mathcal F}(S)$ acts transitively on the remaining $F_i \in \{F_2, \ldots , F_5\}\setminus \{F\}$, we have that either $F$ is the only essential subgroup in $\{F_2, \ldots, F_5\}$ or all $F_i$, $i = 2,\ldots , 5$ are essential. This is (2).
\end{proof}
\section{The fusion system $\mathcal{F}$}\label{sec:F}

Recall that by \fref{prop:essential} any essential subgroup is normal in $S$. We use this in what follows without further saying.
In \fref{prop:normalinS} we listed the potential essential subgroups $F_1, \ldots F_5$ and their automizers. We will keep the notation from this proposition. We first collect the properties of the $F_i$ proved in \fref{sec:essnormal}, in particular with respect to being essential.

In this section it will become important that $O_2(\mathcal F) = 1$. We will use the following proposition
\begin{proposition}\label{prop:O21}{\rm (\cite[Proposition I.4.5]{AKO})} Let ${\mathcal F}$ be a saturated fusion system over a finite $p$-group $P$, and fix $R \leq  P$.
Then $R$ is normal in ${\mathcal F}$  if and only if it satisfies the following condition: if $F = P$ or $F$ is ${\mathcal F}$-essential,
then $F \geq  R$  and each $\alpha \in  \Aut_{\mathcal F}(F)$ sends $R$ to itself.
\end{proposition}

\begin{lemma}\label{lem:normalE} There is a uniquely determined essential subgroup $F_1$ of $S$ with $Q \not\leq F_1$. We have $\Out_{\mathcal F}(F_1) \cong \Sigma_3$ or $\Sigma_3 \times \Z_3$ depending on $\Aut_{\mathcal F}(S) = \Inn(S)$ or not. Further
any elementary abelian subgroup of order $64$  in $S$ is invariant under  $O^{2^\prime}(\Aut_{\mathcal F}(F_1))$.
\end{lemma}

\begin{proof} As  $Q$ is not normal in $\mathcal F$ then by \fref{lem:wc} and \fref{prop:O21} there must be some essential subgroup $F$ with $Q \not\leq F$.  By \fref{lem:essentialnormal2} $F$ is uniquely determined and so $F = F_1$. Further $\Out_{\mathcal F}(F_1) \cong \Sigma_3$ or $\Sigma_3 \times \Z_3$ by \fref{prop:normalinS}(2). The first case occurs if  $\Aut_{\mathcal F}(S) = \Inn(S)$. Hence $O^{2^\prime}(\Out_{\mathcal F}(F_1)) \cong \Sigma_3$. This is the first assertion. The second one follows by \fref{lem:essentialnormal2}(1).
\end{proof}

\begin{lemma}\label{lem:E64} Let $F \in \{F_2, \ldots F_5\}$. Then $Q \leq F$ and $\bar{F}$  is a fours group, which just consists of involutions corresponding to elementary abelian subgroups of order $64$  in the sense of \fref{lem:64}. In particular  $\bar{s}_{i_0} \not\in  \bar{F}$. If $F$ is essential these three elementary abelian groups of order $64$ are all invariant under $O^{2^\prime}(\Aut_{\mathcal F}(F))$.
\end{lemma}

\begin{proof} The assertion is just an application of  \fref{lem:essentialnormal2}(2) and \fref{prop:normalinS}.
\end{proof}

\begin{lemma}\label{lem:essentialgroups}
Among the $F_i$, $i = 2, \ldots , 5$, there are at least three, which are essential in $\mathcal F$.

Let $F_i$, $F_j$ be two of these essential subgroups, then there is some uniquely determined elementary abelian subgroup $E_{ij}$ of order $64$ in $S$ such that $E_{ij}$ is invariant  under $\Aut_{\mathcal F}(F_i)$ and $\Aut_{\mathcal F}(F_j)$. In particular there is some elementary abelian group $E_{ij}$ in $S$ of order $64$ such that  $N_{\mathcal F}(E_{ij}) = \langle O^{2^\prime}(\Aut_{\mathcal F}(F_1)), O^{2^\prime}(\Aut_{\mathcal F}(F_i)), O^{2^\prime}(\Aut_{\mathcal F}(F_j)) \rangle $ and there is a model $G_{E_{ij}}$ for $N_{\mathcal F}(E_{ij})$ with $G_{E_{ij}}/E_{ij} \cong A_8 \cong \Omega^+_6(2)$.
\end{lemma}

\begin{proof} By \fref{lem:normalE} we have that $F_1$ is the uniquely determined essential subgroup of $S$ with does not normalize $Q$. By \fref{lem:essentialnormal2}(1) $Z_2(S)$ is invariant under $N_{\mathcal F}(F_1)$.
As $Z_2(S)$ cannot be normal in $\mathcal F$ by \fref{prop:O21} we must have further essential subgroups $F_i$, which then are among $F_2, \ldots , F_5$. By \fref{lem:essentialnormal2} these all contain $Q$.
\\
\\
Assume there is just one such essential subgroup, call it  $F_2$. Then $\Aut_{\mathcal F}(S)$ normalizes $F_2$. By \fref{lem:E64} and \fref{lem:essentialnormal2}(2) there are exactly three elementary abelian subgroups $E_i$, $i = 1,2,3$, of order 64, which are contained in $F_2$. Thus $\Aut_{\mathcal F}(S)$ permutes them. This shows that $X = \langle E_1, E_2, E_3 \rangle$ is invariant under $\Aut_{\mathcal F}(S)$. By \fref{lem:normalE} also $X$ is invariant under  $O^{2^\prime}(\Aut_{\mathcal F}(F_1))$ and then also under $\Aut_{\mathcal F}(F_1)$. As $\mathcal F = \langle \Aut_{\mathcal F}(S), \Aut_{\mathcal F}(F_1), \Aut_{\mathcal F}(F_2) \rangle$ by \fref{prop:alperin} we see by \fref{prop:O21} that $X$ is normal in $\mathcal F$, a contradiction.
\\
\\
Thus we have at least two essential subgroups $F_i$, $F_j$ in $S$, which contain $Q$.  By \fref{lem:E64} and \fref{lem:essentialnormal2}(2) $|\bar{F}_i \cap \bar{F}_j| = 2$,  and so there is exactly one elementary abelian subgroup $E_{ij}$ of order 64 which is contained in $F_i \cap F_j$.

We now will  prove the first claim, i.e. there are at least three essential sugroups in $\{F_2, \ldots , F_5\}$.  If $3$ divides the order  of $\Aut_{\mathcal F}(S)$, the claim follows by \fref{prop:normalinS}(2).

Now assume that $\Aut_{\mathcal F}(S) = \Inn(S)$. Thus $E_{ij}$ is  invariant under $\Aut_{\mathcal F}(F_i)$ and $\Aut_{\mathcal F}(F_j)$ as well.  By \fref{lem:normalE}  $\Aut_{\mathcal F}(F_1)$ leaves $E_{ij}$ invariant. But as  $E_{ij}$ is not normal in $\mathcal F$, we  receive by \fref{prop:O21} that $\mathcal F$ is not generated by  $\Aut_{\mathcal F}(F_1)$, $\Aut_{\mathcal F}(F_i)$ and  $\Aut_{\mathcal F}(F_j)$. Thus in any case by \fref{prop:alperin} there are at least three essential subgroups containing $Q$.  This proves the first claim.
\\
\\
 By \fref{lem:normalE} $E_{ij}$ is invariant under $O^{2^\prime}(\Aut_{\mathcal F}(F_1))$ as well. We will consider $N_{\mathcal F}(E_{ij})$. By \cite[Theorem 2.1]{AKO}  $N_{\mathcal F}(E_{ij})$ is saturated. Now $F_1$, $F_i$, $F_j$ are all essential in $N_{\mathcal F}(E_{ij})$ and obviously these are the only ones. Hence we receive that   $O^{2^\prime}(\Aut_{\mathcal F}(F_i))$ , $O^{2^\prime}(\Aut_{\mathcal F}(F_j))$  and $O^{2^\prime}(\Aut_{\mathcal F}(F_1))$ are all contained   $N_{\mathcal F}(E_{ij})$.    By \fref{lem:aut1} we have that $Aut_{N_{\mathcal F}(E_{ij})}(S) = \Inn(S)$. Hence by  \fref{prop:alperin}  we have
 $$N_{\mathcal F}(E_{ij}) = \langle O^{2^\prime}(Aut_{\mathcal F}(F_i)),  O^{2^\prime}(Aut_{\mathcal F}(F_j)), O^{2^\prime}(Aut_{\mathcal F}(F_1))\rangle.$$

  By the model theorem \cite[Theorem 4.9]{AKO} there is a unique model $G_{E_{ij}}$ for $N_{\mathcal F}(E_{ij})$. Set  $\hat{\mathcal F} := N_{\mathcal F}(E_{ij})/E_{ij}$. Then $\hat{\mathcal F}$ is the 2-fusion system of $G_{E_{ij}}/E_{ij}$ and so it is also saturated. We show
  \begin{gather*}\label{eq:O2X} O_2(\hat{{\mathcal F}})  = 1.\tag{$\ast$}\end{gather*}

 We consider $QE_{ij}/E_{ij}$, which is elementary abelian of order 16 (see the statement before \fref{lem:A8}). By \fref{lem:64} for any involution  $y \in S/E_{ij} \setminus QE_{ij}/E_{ij}$ we have that $|[QE_{ij}/E_{ij},y]| = 4$. This implies that $N_{G_{E_{ij}}}(QE_{ij}/E_{ij})$ acts irreducibly on $QE_{ij}/E_{ij}$. Hence if $O_2(\hat{{\mathcal F}}) \not= 1$, then  $O_2({\hat{\mathcal F}})  =QE_{ij}/E_{ij}$. But by the choice of $F_1$ we have  $Q \not\leq F_1$ and so by \fref{prop:O21} \fref{eq:O2X} holds.
\\
\\
As the sectional 2-rank of $S/E$ is four and \fref{eq:O2X} holds, we may apply \cite[Theorem A (5)]{Oliv}, which gives that
$\hat{\mathcal F}$ is the 2-fusion system of some $\PSp_4(q)$, $q$ odd, or $A_8$. Recall that in the first systems $\Aut_{\hat{\mathcal F}}(QE_{ij}/E_{ij})$ would induce $A_5$, which is not the case. Thus we have the latter. This now shows that  $G_{E_{ij}}/E_{ij}$ is a subgroup of $\GL_6(2)$, with a fusion system of type $A_8$.

By \cite[Theorem A]{GoHa} $G_{E_{ij}}/E_{ij} \cong A_8$ or $A_9$. As $A_9$ contains an elementary abelian subgroup of order 9 all of whose elements of order three are conjugate, coprime action shows that the smallest $\GF(2)$-representation for $A_9$ is of dimension 8. Thus we have
$G_{E_{ij}}/E_{ij} \cong A_8 \cong \Omega^+_6(2)$.
\end{proof}

We will use \fref{lem:essentialgroups} and \cite{Orno} to produce a parabolic system for the fusion system $\tilde{\mathcal F}$ to be defined below. The problem with the application of \cite{Orno} is that it requires that $N_{\mathcal F}(S)$ is contained in $N_{\mathcal F}(F)$ for any essential subgroup $F$. But this is not the case as long as $\Aut_{\mathcal F}(S) \not= \Inn_{\mathcal F}(S)$. We define $$\tilde{\mathcal F} = \langle O^{2^\prime}(\Aut_{\mathcal F}(F_1)), O^{2^\prime}(\Aut_{\mathcal F}(F))~|~Q \leq F \le S, F~\mathcal F-\text{essential}\rangle.$$

\begin{lemma}\label{lem:subsystem}  $\tilde{\mathcal F} = O^{2^\prime}(\mathcal F)$ is saturated with the same essential subgroups as in $\mathcal F$ and $\Aut_{\tilde{\mathcal F}}(S) = \Inn(S)$.
\end{lemma}

\begin{proof}  That the essential subgroups in $\mathcal F$ and $\tilde{\mathcal F}$ are the same is obvious. Let $F$ be some essential subgroup of $\mathcal F$. Then
 $$\Aut^{(F)}_{\mathcal F}(S) := \{\alpha \in \Aut_{\mathcal F}(S)~|~\alpha(F) = F, \alpha_{| F} \in O^{2^\prime}(\Aut_{\mathcal F}(F))\} = \Inn(S)$$ by \fref{prop:normalinS}. Hence by \cite[Lemma 1.4]{Oliv1} we have that
 $O^{2^\prime}(\mathcal F) \not= \mathcal F$ and so $\tilde{\mathcal F} = O^{2^\prime}(\mathcal F)$. By \cite[Theorem 7.7]{AKO} $\tilde{\mathcal F}$ is saturated. Clearly $\Inn (S) = \Aut_{\tilde{\mathcal F}}(S)$.
 \end{proof}

For the convenience of the reader we repeat the following definition:
\begin{Def}\label{def:Ono}{\rm \cite[Definition 5.1]{Orno})} {\rm Let $\mathcal F$ be a fusion system over a $p$-group $S$ and set ${\mathcal B} = N_{\mathcal F}(S)$. We say that $\mathcal F$ has a family of parabolic subsystems if $\mathcal F$ contains a collection $\{{\mathcal F}_i~|~i\in I\}$ of saturated constrained fusion subsystems, each one with exactly one ${\mathcal F}_i$-conjugacy class of ${\mathcal F}_i$-essential subgroups, with the following properties
\begin{itemize}
\item[(1)] $\mathcal B$ is a proper subsystem of ${\mathcal F}_i$ for all $i \in I$;
\item[(2)] ${\mathcal F} = \langle {\mathcal F}_i~|~ i \in I\}$and no proper subset $\{{\mathcal F}_j~|~j \in J \subset I\}$ generates $\mathcal F$;
\item[(3)]  ${\mathcal F}_i \cap {\mathcal F}_j = \mathcal B$ for any pair of distict elements ${\mathcal F}_i$ and ${\mathcal F}_j$;
\item[(4)] ${\mathcal F}_{ij} : = \langle {\mathcal F}_i, {\mathcal F}_j \rangle$ is a saturated constrained subsystem of $\mathcal F$ for all $i,j \in I$.
    \end{itemize}}
    \end{Def}

    Set

\begin{lemma}\label{lem:parabolic} In $\tilde{\mathcal F}$ consider the collection of  subsystems $$\{O^{2^\prime}(N_{\mathcal F}(F_1),  O^{2^\prime}(N_{\mathcal F}(F))~|~Q \leq F \le S, F~\mathcal F\text{-essential}\}.$$ Then this collection satisfies  \fref{def:Ono}$(1)$, $(3)$ and $(4)$.
\end{lemma}

\begin{proof} Obviously all of $O^{2^\prime}(N_{\mathcal F}(F_1)$ and   $O^{2^\prime}(N_{\mathcal F}(F))$, $F$ essential with $Q \leq F$, are constrained saturated fusion systems,which have exactly one essential subgroup $F_1$, $F$, respectively.

We have ${\mathcal B} = N_{\tilde{\mathcal F}}(S) = {\mathcal F}_S(S)$ according to \fref{lem:subsystem}. Thus  \fref{def:Ono}(1)  is satisfied.

As for all essential subgroups $U$ of $S$ we have $O^{2^\prime}(\Out_{\mathcal F}(U)) \cong \Sigma_3$ by \fref{lem:essentialnormal2} \fref{def:Ono}(3) holds.
\\
\\
We now show that also \fref{def:Ono}(4) holds. Take two of the subsystems above and call them ${\mathcal G}$ and ${\mathcal H}$. Then by \fref{lem:essentialgroups} they are contained in a set of three subsystems, which generate the normalizer in $\mathcal F$ of an elementary abelian group $E$ of order 64. Further the fusion system $\langle \mathcal G, \mathcal H\rangle$ is the fusion system of some parabolic subgroup of $G_E$, the model for $N_{\mathcal F}(E)$ as in \fref{lem:essentialgroups}. Hence it is saturated.  Thus also  \fref{def:Ono}(4) is satisfied.
\end{proof}

We now introduce the following notation.
As in \fref{lem:normalE} we fix $F_1$ the essential subgroup of $S$ with $Q \not\leq F_1$ and further three essential subgroups named $F_2$, $F_3$, $F_4$.
Set ${\mathcal F}_i$, $i = 1,2,3,4$, as $O^{2^\prime}(N_{\mathcal F}(F_i))$. Then we consider the following fusion system on $S$:
$$\bar{\mathcal F} = \langle {\mathcal F}_i, i = 1,2,3,4 \rangle.$$
We have that $\bar{\mathcal F} \subseteq \tilde{\mathcal F}$. We do not claim that $\bar{\mathcal F}$ is saturated. To prove this in fact is the aim of the paper \cite{Orno}. But we can do all the calculation within the saturated fusion system $O^{2^\prime}(\mathcal F)$.

\begin{lemma}\label{lem:Orno} The collection $\{{\mathcal F}_i$, $i = 1,2,3,4\}$, is a family of parabolic subsystems in $\bar{\mathcal F}$ in the sense of \fref{def:Ono}.
\end{lemma}

\begin{proof} By \fref{lem:parabolic} we just have to show that \fref{def:Ono}(2) holds. Hence we have to show that $\bar{\mathcal F}$ is not generated by three of the ${\mathcal F}_i$.

Consider first $\langle {\mathcal F}_2, {\mathcal F}_3, {\mathcal F}_4 \rangle = \mathcal H$.
By \fref{lem:wc} ${\mathcal H}$ leaves  $Q$ invariant and we are done as ${\mathcal F}_1$ does not.

Thus we may consider $\langle {\mathcal F}_1, {\mathcal F}_2, {\mathcal F}_3 \rangle$.  But then by \fref{lem:essentialgroups} they all normalize some elementary abelian group $E$ of order 64. In particular $F_1, F_2, F_3$ are  essential subgroups in $G_E$, $G_E$ the model for $N_{\mathcal F}(E) = N_{\bar{\mathcal F}}(E)$. As $G_E$ has just three minimal parabolics, there are also exactly three of the $F_j$ essential in $N_{\mathcal F}(E)$ and so ${\mathcal F}_4 \not\subseteq N_{\mathcal F}(E)$, which proves the lemma.
\end{proof}

By \cite[Proposition 5.2]{Orno} there are 2-constrained finite groups $G_i$ and $G_{ij}$  which realize ${\mathcal F}_i$ and ${\mathcal F}_{ij} = \langle {\mathcal F}_i, {\mathcal F}_j \rangle$, $i,j = 1,2,3,4.$ Further there are injective group homomorphisms $\psi_{ij} :G_i \rightarrow G_{ij}$, $\psi_i : S \rightarrow G_i$, such that $\mathcal A = \{(S;G_i,G_{ij}), (\psi_i, \psi_{ij}), i,j = 1,2,3,4\}$ is a diagram of groups. Let $G$ be a faithful completion of $\mathcal A$. Then by \cite[Lemma 5.4]{Orno} $\{S, G_i, i = 1,2,3,4\}$ is a parabolic system of $G$.

We have that $G_i/O_2(G_i) \cong \Sigma_3$. If we apply \fref{lem:essentialgroups} to $\bar{\mathcal F}$ as in the proof before we receive that $G_{ij} \leq G_{E_{ij}}$ with $G_{E_{ij}}/E_{ij} \cong \Omega^+_6(2)$. As $G_{ij}$ is a maximal parabolic we receive  either $G_{ij}/O_2(G_{ij}) \cong \GL_3(2)$ or $\Sigma_3 \times \Sigma_3$. Hence $\{G_1,G_2,G_3,G_4\}$ forms a classical parabolic system of type $\mathcal M$ in the sense of \cite[Definition 7.4]{Orno} with the following diagram:

$$\Dvier$$

We now receive
\begin{proposition}\label{prop:O82} $\bar{\mathcal F}$ is the $2$-fusion system of $\Omega^+_8(2)$. In particular if $\mathcal F$ possesses exactly four essential subgroups, then $\bar{\mathcal F} = \tilde{\mathcal F}$ is the $2$-fusion system of $\Omega^+_8(2)$.
\end{proposition}

\begin{proof} By \cite[Proposition 7.5]{Orno} $\bar{\mathcal F}$ is the  $2$-fusion system of some group $G$ of Lie type in characteristic two, which has the Dynkin diagram above, extended by diagonal- and/or field-automorphisms. (The proof of \cite[Proposition 7.5]{Orno} involves the result \cite{Tim}) As we have a $D_4$-diagram we get that this group is $D_4(2^m)$ for suitable $m$. As $|Z(S)| =2$, we then have that $m = 1$. In particular there are no diagonal or field automorphisms.  Hence $G \cong \Omega^+_8(2)$.
\end{proof}

We now consider the case that we have four essential subgroups of $S$ containing $Q$. Hence we extend the notation above. We denote by $F_5$ the fourth one which contains $Q$ and set ${\mathcal F}_5 = O^{2^\prime}(N_{\mathcal F}(F_5))$. Then set $$\tilde{\mathcal F} = \langle {\mathcal F}_i, i = 1, \ldots 5 \rangle = O^{2^\prime}(\mathcal F)$$
We approach as before.

\begin{lemma}\label{lem:Orno2} $\{{\mathcal F}_i, i = 1, \ldots , 5\}$ is  a family of parabolic subsystems in $\tilde{\mathcal F}$.
\end{lemma}

\begin{proof} By \fref{lem:parabolic} we just have to show that \fref{def:Ono}(2) holds.
For this we consider a collection of four of the ${\mathcal F}_i$. If ${\mathcal F}_1$ is not among them, we have by \fref{lem:wc} that they all leave $Q$ invariant  and so as ${\mathcal F}_1$ does not, they cannot generate $\tilde{\mathcal F}$. Thus we may assume that ${\mathcal F}_1$ is one of them. Changing numbering, we may assume that they generate $\bar{\mathcal F}$. According to \fref{prop:O82} $\bar{\mathcal F}$ is the 2-fusion system of $\Omega^+_8(2)$. But in $\Omega^+_8(2)$ we have elementary abelian subgroups of order 64 which are not normalized by $\Omega^+_6(2)$, which contradicts  \fref{lem:essentialgroups}. Thus \fref{def:Ono}(2) is satisfied.
\end{proof}

Again by \cite[Proposition 5.2]{Orno} there are groups $G_i$ and $G_{ij}$,  which realize ${\mathcal F}_i$ and ${\mathcal F}_{ij} = \langle {\mathcal F}_i, {\mathcal F}_j \rangle$, $i,j = 1,2,3,4,5$. Further the $G_i$ form a parabolic system. We have that $G_i/O_2(G_i) \cong \Sigma_3$. As before by \fref{lem:essentialgroups} we receive that  either $G_{ij}/O_2(G_{ij}) \cong \L_3(2)$ or $\Sigma_3 \times \Sigma_3$. Hence $\{G_1,G_2,G_3,G_4,G_5\}$ forms a classical parabolic system of type $\mathcal M$ in the sense of \cite[Definition 7.4]{Orno}. But this time the diagram is not spherical and so we cannot apply \cite[Proposition 7.5]{Orno}.

The diagram is as follows:
$$\setlength{\unitlength}{2mm}
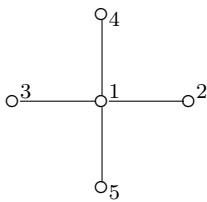
\begin{picture}(50,40)(-20,-30) \put(0.3,0.3){\makebox(0,0){$\circ^1$}} \put(6.1,0.3){\makebox(0,0){$\circ^2$}}
\put(-5.6,0.3){\makebox(0,0){$\circ^3$}} \put(0.3,5.6){\makebox(0,0){$\circ_4$}}
\put(0.3,-5.9){\makebox(0,0){$\circ_5$}} \put(-0.05,0.4){\line(0,1){5}} \put(-0.05,-0.4){\line(0,-1){5}}
\put(-0.5,0){\line(-1,0){5}} \put(0.4,0){\line(1,0){5}} \end{picture} $$
\vspace{-4cm}

But we can follow the proof of \cite[Proposition 7.5]{Orno}. In the first place we have to have control over the universal group for our parabolic system (chamber system).
\begin{proposition}\label{prop:covering} Let $\mathcal C$ be the chamber system corresponding to the parabolic system above. Then
the universal covering is the affine building $\Gamma$ of type $\Omega_8(\Q_2,f)$ where $f$ is the form $f = \sum_{i=1}^8 x_i^2$. The discrete group acting on the building is $\Omega_8(\Z[\frac{1}{2}],f)$.
\end{proposition}

\begin{proof} This is \cite[Theorem 7.1 and Corollary 7.4]{Kant}.
\end{proof}

\begin{lemma}\label{lem:sylowOmega} We have that $S$ is a Sylow $2$-subgroup of $\Omega_8(\Z[\frac{1}{2}],f)$.
\end{lemma}

\begin{proof} According to \cite[2.11]{Orno} we just have to show that any finite $2$-subgroup $P$ of $G= \Omega_8(\Z[\frac{1}{2}],f)$ can be conjugate into $S$. By the Bruhat-Tits fixed point theorem \cite[(3.2.4)]{IHES} $P$ has a fixed point in the building $\Gamma$. Hence it fixes a proper residuum $R$ in $\Gamma$, which is a finite building over a field in characteristic 2. By Sylows theorem $P$ is conjugate into some Sylow $2$-subgroup of the stabilizer of $R$. We have $G_R \cong  \Omega_8^+(2)$ and we may assume that $G_R$ contains a conjugate of $S$ and so $P$ is conjugate into $S$, the assertion.
\end{proof}

We will use \cite[Proposition 5.6]{Orno}. For this we have to show

\begin{lemma}\label{lem:connected}  Let ${\mathcal C}$ be the chamber system of the building $\Gamma$ from \fref{prop:covering}. Then for any finite $2$-subgroup $P$ the chamber system ${\mathcal C}^P$ of fixed points of $P$ is contractible.
\end{lemma}

\begin{proof} By \fref{lem:sylowOmega} we have that $P$ is conjugate into $S$. Hence first of all we may assume that $P \leq S$ and further $P$ stabilizes a chamber in some residue $R$. As chambers in $R$ are also chambers in $\Gamma$ we have that ${\mathcal C}^P \not= \emptyset$. Now the proof of connectedness and contractibility follows the lines of  part (b) of the proof of \cite[Theorem 3.1]{Qui}, which just is based on the fact that in the natural geometric realization of $\Gamma$ there is a unique minimal geodesic  between two chambers.
\end{proof}

\begin{proposition}\label{prop:o83} We have that $\tilde{\mathcal F}$ is the $2$-fusion system of $\Omega_8(\Z[\frac{1}{2}],f)$. In particular it is the $2$-fusion system of $\POmega^+_8(3)$.
\end{proposition}

\begin{proof}  By \fref{lem:connected} we may  apply \cite[Proposition 5.6]{Orno}. This yields the first assertion.

By \cite{Kneser} we know the finite factors of $\Omega_8(\Z[\frac{1}{2}],f)$. The simple ones are $\POmega^+_8(p)$ for some suitable prime $p$. Hence $\POmega^+_8(3)$ is one of them, which also can be seen by \cite{AschSm}.
\end{proof}

We should remark, that according to the proof before all $\POmega^+_8(p)$, with $p \equiv \pm 3 ( mod~8)$ are factors of $\Omega_8(\Z[\frac{1}{2}],f)$. Hence they all have the fusion system $\tilde{\mathcal F}$.

If $\mathcal F = \tilde{\mathcal F}$ we are done. This is exactly in case of $\Aut_{\mathcal F}(S) = \Inn(S)$.
Let us collect:
\begin{proposition}\label{prop:InnS} Let $\mathcal F$ be a saturated fusion system on a Sylow $2$-subgroup $S$ of $\Omega_8^+(2)$ with $O_2(\mathcal F) = 1$ and $\Aut_{\mathcal F}(S) = \Inn (S)$, then $\mathcal F$ is a $2$-fusion system of $\Omega_8^+(2)$ or $\POmega_8^+(3)$.
\end{proposition}
Hence from now on we will assume that $3$ divides $|\Aut_{\mathcal F}(S)|$. We have
$$\mathcal F = \langle \tilde{\mathcal F}, N_{\mathcal F}(S) \rangle.$$
Further $\tilde{\mathcal F} = O^{2^\prime}(\mathcal F)$. This in particular proves the main theorem.

\section{Appendix}
If $\mathcal F$ is a fusion system on $S$ with $O_2(\mathcal F) = 1$ and $\Aut_{\mathcal F}(S) = \Inn(S)$, we have that $\mathcal F$ is either the 2-fusion system of $\Omega_8^+(2)$ or $\POmega_8^+(3)$. According to \fref{lem:aut} it remains to deal with the case that $3$ divides  $|\Aut_{\mathcal F}(S)|$. There are $2$-fusion systems on $\Omega_8^+(2):3$ or $\POmega_8^+(3):3$, where always a graph automorphism is induced. There is a hint that these might be the only ones with $\Aut_{\mathcal F}(S) \not= \Inn(S)$. \fref{lem:O2prim} gives some indication for this but not a proof.

The question will be if $\mathcal F$ is uniquely determined. This can also be stated in the following way.
Is an automorphism $\rho \in \Aut_{\mathcal F}(S)$ of order three  unique up conjugacy in $\Aut(S)$. By \fref{lem:aut} a Sylow $3$-subgroup of $\Aut(S)$ is of order three. Hence all subgroups of order three in $\Aut(S)$ are conjugate. Thus up to conjugacy $\mathcal F$ is uniquely determined.  But this is not something one can generalize to the case of $S$ a Sylow $p$-subgroup of $\POmega^+_8(p^{a})$, even not for $p = 2$. As Sylow $r$-subgroups for $r \not= p$ not necessarily are cyclic. If we think about the conjecture stated in the introduction we will also meet with the problem of graph automorphisms of order 2 for $p$ odd.

But there is a different way to approach, which we will sketch now.
\\
\\
The first question will be if $\mathcal F$ can be exotic. In fact as seen $O^{2^\prime}(\mathcal F)$ is not. Here a beautiful result due to C. Broto, J. Moller, B. Oliver and A. Ruiz \cite[Corollary D]{BMOR} helps. For our case  it reads as follows:
\begin{proposition}\label{prop:exotic} Let $\mathcal F$ be a fusion system on a $p$-group $S$. Assume that $O_p(\mathcal F) = 1$ and $O^p(O^{p^\prime}(\mathcal F)) = O^{p^\prime}(\mathcal F)$. If $O^{p^\prime}(\mathcal F)$ is realizable by a finite group, then also $\mathcal F$ is realizable.
\end{proposition}
If we change the conjecture to {\it there are no exotic fusion systems $\mathcal F$ with $O_p(\mathcal F)= 1$ and $O^{p^\prime}(\mathcal F) = O^p(O^{p^\prime}(\mathcal F))$ on a Sylow $p$-subgroups of groups of Lie type in characteristic $p$ if rank or field is large enough}, then by \fref{prop:exotic}  we just have to prove it for $\mathcal F = O^{p^\prime}(\mathcal F)$.
\\
\\
We just have answered our first question that $\mathcal F$ is realizable by a finite group $G$. If this has a normal subgroup of index three we are done by \fref{prop:InnS}. Now as usual we can reduce it to the question if there is some  finite simple group $G$ with Sylow $2$-subgroup $S$ isomorphic to the Sylow 2-subgroup of $\Omega^+_8(2)$,
and with $|N_G(S)| = 3|S|$. As the proof of  \cite[Corollary D]{BMOR} uses the classification of the finite simple groups we also can do so.

It is easy to see that $G$ cannot be sporadic or alternating, as no such group has a Sylow 2-subgroup of order $2^{12}$ (see \cite[Table 5.3]{GLS}).

Assume that $G$ is of Lie type, i.e. $G = G(q)$, $q$ a prime power. For the order of $G$ we use \cite[Table 2.2]{GLS}. Assume first that $q$ is even. As the center of a Sylow 2-subgroup of $G(q)$ contains a root subgroup of order $q$ and $|Z(S)| = 2$, we have $q = 2$. Now by $|S| = 2^{12}$ we see that $G \cong \Omega^\pm_8(2)$ or ${}^3D_4(2)$. We know that in $\Omega_8^+(2)$ $N_G(S) = S$. In $G = {}^3D_4(2)$ we have that $N_G(Q)/Q \cong \L_2(8)$ and so $7$ divides $|N_G(S)|$, which contradicts \fref{lem:aut}. We are left with $G = \Omega^-_8(2)$. Then $N_G(Q)/Q \cong \Omega^-_4(2) \times \Sigma_3$. Now let $\rho \in N_G(S)$ be of order three. Then by \fref{lem:aut1} $C_{S/Q}(\rho) = \langle s_{i_0} \rangle$. Hence $\langle s_{i_0} \rangle$ is centralized by $\Omega^-_4(2)$. This means that $[Q/Z(Q),s_{i_0}]$ is the natural $\Omega^-_4(2)$-module and so the preimage must be elementary abelian. But this contradicts \fref{lem:64}(b).(For short, the Sylow 2-subgroups of $\Omega^+_8(2)$ and $\Omega^-_8(2)$ are not isomorphic.)
\\
\\
\begin{tabular}{c|l}
$G$&$v_2(|G|)$, $q$ odd\\
\hline\\
$\PGL_7(q)$&$v_2((q^7-1)(q^6-1)(q^5-1)(q^4-1)(q^3-1)(q^2-1)) > 12$\\
$\PSU_7(q)$&$v_2((q^7+1)(q^6-1)(q^5+1)(q^4-1)(q^3+1)(q^2-1)) > 12$\\
$\PSp_8(q)$&$v_2((q^8-1)(q^6-1)(q^4-1)(q^2-1))-1 > 12$\\
$\POmega_9(q)$&$v_2(|PSp_8(q)|) > 12$\\
$\POmega_8^-(q)$&$v_2((q^6-1)(q^4-1)(q^2-1)(q^4+1)) -1=3\cdot v_2(q^2-1) +1 \not= 12$\\
$\POmega_7(q)$&$ v_2((q^2-1)(q^4-1)(q^6-1)) -1 = 3\cdot v_2(q^2-1)$\\
$\POmega^+_8(q)$&$v_2((q^2-1)(q^4-1)(q^6-1)(q^4-1))-2 = 4\cdot v_2(q^2-1)$\\
$G_2(q)$&$v_2((q^2-1)(q^6-1)) = 2\cdot v_2(q^2-1)$\\
$F_4(q)$&$v_2((q^2-1)(q^6-1)(q^8-1)(q^{12}-1)) > 12$\\
$E_6(q)$&$v_2(|E_6|) > v_2(|F_4(q)|)$\\
$E_7(q)$&$v_2(|E_7|) > v_2(|F_4(q)|)$\\
$E_8(q)$&$v_2(|E_8|) > v_2(|F_4(q)|)$\\
${}^2E_6(q)$&$v_2(|{}^2E_6(q)|) > v_2(|F_4(q)|)$\\
${}^3D_4(q)$&$v_2((q^2-1)(q^6-1)(q^8+q^4+1)) = 2 \cdot v_2(q^2-1)$\\
\end{tabular}
\vspace{0.5cm}

Now consider $q$ odd. Then we use $v_2(q^2-1)$ and $v_2(q-1)$ to determine the 2-part of $|G|$. The table above shows that in the exceptional case we have $G = G_2(q)$ or ${}^3D_4(q)$ and in both case $v_2(q^2-1) = 6$. Recall that ${}^2G_2(q)$ has elementary an abelian Sylow 2-subgroup.

 Assume that $G$ is classical. Now consider for a moment the universal version $\hat{G}$. This then acts on the natural vector space. In $G$ we have an elementary abelian group of order 64. The preimage in $\hat{G}$ will be elementary abelian, extraspecial of order $2^7$ or a central product of a cyclic group of order 4 by an extraspecial group of order $2^7$. In the first case the minimal faithful representation of such a group has dimension 7, while in the remaining cases the dimension will be 8. Now we use the order formulas as shown above. Recall that $v_2(q^2-1) \geq 3$.
\\
\\This implies that we just have four possibilities for $G$, $G \cong G_2(q)$, ${}^3D_4(q)$, $\POmega_7(q)$ or $\POmega^+_8(q)$. In the first two cases we have $v_2(q^2-1) = 6$, in the third
 $v_2(q^2-1) = 4$, while in the last case $v_2(q^2-1) = 3$.

 We will use \cite{Asch}. We consider the set $\Fun(S)$. This is the set of groups $K(S) \cong \SL_2(q)$ generated by two root subgroups of order $q$ such that  $K(S) \cap S$ is a Sylow 2-subgroup of $K(S)$. We have that distinct elements in $\Fun(S)$ commute. Hence in the first two cases we have $SL_2(q) \leq G$ and so some quaternion group of order $2^6$ is contained in $S$, which is impossible.

 Thus we are left with $G \cong \POmega_7(q)$ or $\POmega_8^+(q)$.
 We next show
 \begin{gather*}\label{eq:rhotrivial} \text{If }\rho \in \Aut_G(S), o(\rho) = 3, \text{ then }|C_{S/Q}(\rho)| \geq 4.\tag{$\ast$}\end{gather*}
Suppose we have shown \fref{eq:rhotrivial}, then as $|S/Q| = 8$, we have that $[\rho,S] \leq Q$, which contradicts \fref{lem:aut1}.
Hence all we have to do is to prove \fref{eq:rhotrivial}.

Again we study $\Fun(S)$. Obviously $\Fun(S)$ is invariant under $N_G(S)$.
In the first case by \cite[Theorem 2]{Asch} $|\Fun(S)| = 2$, while in the second one $|\Fun(S)| = 4$.

 By \cite[Theorem 2]{Asch} we have $N_G(\Fun(S))$ induces on $\Fun(S)$ a group of order two, four, respectively.  Hence our element $\rho \in \Aut_G(S)$ must normalize any $K(S) \in \Fun(S)$. If $G \cong \POmega^+_8(q)$, then $\langle K(S) \cap S~|~K(S) \in \Fun(S) \rangle = Q$. Thus $\rho$ centralizes the fours group in $S/Q$, which acts transitively on $\Fun(S)$ and so \fref{eq:rhotrivial} holds.

Let now $G = \POmega_7(q)$. Then again $\rho$ acts on each $K(S) \cap S$. But as $v_2(q^2-1) = 4$, $K(S) \cap S$ is quaternion of order 16 and so $\rho$ must centralize $K(S) \cap S$. As $Q$ does not contain quaternion subgroups of order 16, we have $|K(S) \cap S : K(S) \cap Q|\geq 2$ and so $|\langle K(S) \cap S~|~K(S) \in \Fun(S) \rangle : \langle K(S) \cap S~|~K(S) \in \Fun(S) \rangle \cap Q|\geq 4$. But this factor group is centralized by $\rho$ and so again \fref{eq:rhotrivial} holds.
\\
\\
We have shown
\begin{lemma}\label{lem:simple} There is  no finite simple group $G$ with Sylow $2$-subgroup $S$ of type $\Omega_8^+(2)$ and $|\Out_G(S)| = 3$.
\end{lemma}

By \fref{prop:InnS} and \fref{prop:exotic} we know  that in case of $|\Out_{\mathcal F}(S)| = 3$ the saturated fusion system $\mathcal F$  is realizable. Let $G$ be some finite group, which realizes $\mathcal F$, then by \fref{lem:simple} $G$ has a normal subgroup $H$ of index three, which realizes $O^{2^\prime}(\mathcal F)$. Hence we get

\begin{proposition}\label{prop:realization}  Let  $\mathcal F$ be a fusion system on a Sylow $2$-subgroup of $\Omega_8^+(2)$ with $O_2(\mathcal F) = 1$.  Then $\mathcal F$ is the $2$-fusion system of a group $H$, where either $\Omega_8^+(2) \leq H \leq \Omega^+_8(2):3$ or  $\POmega_8^+(3) \leq H \leq \POmega^+_8(3):3$.
\end{proposition}

\begin{proof} If $\Inn(S) = \Aut_{\mathcal F}(S)$ the result follows by \fref{prop:InnS}. Thus by \fref{lem:aut} we may assume that $|\Out_{\mathcal F}(S)| = 3$.

The essential subgroups of $\mathcal F$ are given by \fref{lem:essentialgroups}. By \fref{prop:exotic} $\mathcal F$ is realizable by a finite group $H$. We have $O(H) = 1$. Further by \fref{lem:essentialgroups} $H$ has two subgroups, which are extensions of an elementary abelian subgroup of order 64 by $A_8$. Hence $O_2(H) = 1$. Thus $F^\ast(H)$ is a direct product of simple groups. Again by \fref{lem:essentialgroups} we have that $F^\ast(H)$ is simple. By \fref{lem:simple} $F^\ast(H) \not= H$.
By the Frattini argument $H = F^\ast(H)N_H(S)$. In particular $|H : F^\ast(H)| = 3$. Further $\Aut_{F^\ast(H)}(S) = \Inn(S)$ and so by \fref{prop:InnS} $F^\ast(H)$ has the 2-fusion system of $\Omega^+_8(2)$ or $\POmega^+_8(3)$. This proves the proposition.
\end{proof}

The notation $\POmega_8^+(3):3$ is not really uniquely defined. We have that  $\Out(\POmega^+_8(3)) \cong \Sigma_4$. Hence  there are four groups $\POmega^+_8(3) : 3$, which are conjugate in $\Aut(\POmega^+_8(3))$. In particular their 2-fusion systems are isomorphic.

 \end{document}